\documentclass[12pt]{amsart}
\usepackage{amsmath}
\usepackage{amscd}
\usepackage{amssymb}
\usepackage{amsfonts}
\usepackage{amsthm}
\usepackage{bbm}
\usepackage{cancel}
\usepackage{color}
\usepackage{eucal}
\usepackage{enumerate,yfonts}

\usepackage{epsfig}
 \usepackage{float}

\usepackage{pdfsync}
\usepackage[all,cmtip]{xy}
\usepackage{latexsym}
\usepackage{mathrsfs}
 \usepackage{placeins}
\usepackage{bookmark}
\usepackage{hyperref}
\usepackage{stmaryrd}
\usepackage{url}

\newtheorem{thm}{Theorem}[section]
\newtheorem{corollary}[thm]{Corollary}
\newtheorem{lemma}[thm]{Lemma}

\newtheorem{proposition}[thm]{Proposition}
\newtheorem{prop}[thm]{Proposition}

\newtheorem{thm-dfn}[thm]{Theorem-Definition}

\theoremstyle{definition}

\newtheorem{remark}[thm]{Remark}
\newtheorem{example}[thm]{Example}

\newtheorem{fact}[thm]{Fact}

\numberwithin{equation}{section}

\newcommand{\fg}{{\mathfrak g}}

\newcommand{\fa}{{\mathfrak a}}

\newcommand{\fc}{{\mathfrak c}}

\newcommand{\fF}{{\mathfrak{F}}}

\newcommand{\Lg}{{\mathfrak g}}

\newcommand{\rW}{{\mathrm W}}

\newcommand{\bC}{{\mathbb C}}

\newcommand{\bZ}{{\mathbb Z}}

\newcommand{\bbP}{{\mathbb P}}

\newcommand{\bP}{{\mathbb P}}

\newcommand{\mE}{\mathcal{E}}
\newcommand{\mF}{\mathcal{F}}

\newcommand{\mO}{\mathcal{O}}
\newcommand{\mL}{\mathcal{L}}
\newcommand{\mH}{\mathcal{H}}

\newcommand{\mN}{\mathcal{N}}
\newcommand{\mB}{\mathcal{B}}

\newcommand{\calF}{{\mathcal F}}

\newcommand{\calL}{{\mathcal L}}

\newcommand{\cB}{{\mathcal B}}

\newcommand{\cO}{{\mathcal O}}

\newcommand{\cF}{{\mathcal F}}
\newcommand{\cN}{{\mathcal N}}

\newcommand{\cP}{{\mathcal P}}
\newcommand{\cE}{{\mathcal E}}

\newcommand{\on}{\operatorname}

\newcommand{\tQ}{\widetilde Q}

\newcommand{\ra}{\rightarrow}

\newcommand{\is}{\simeq}

\newcommand{\Loc}{\on{LocSys}}

\newcommand{\quash}[1]{}  
\newcommand{\nc}{\newcommand}

\nc{\al}{{\alpha}} \nc{\be}{{\beta}} \nc{\ga}{{\gamma}}
\nc{\ve}{{\varepsilon}} \nc{\Ga}{{\Gamma}} 
\nc{\La}{{\Lambda}}

\nc{\ad }{{\on{ad }}}

\nc{\aff}{{\on{aff}}} \nc{\Aff}{{\mathbf{Aff}}}

\nc{\der}{{\on{der}}}

\nc{\diag}{{\on{diag}}}

\nc{\Fl}{{\calF\ell}}

\nc{\Hg}{{\on{Higgs}}}

\newcommand{\id}{{\on{id}}}
\nc{\Id}{{\on{Id}}}

\nc{\Ind}{{\on{Ind}}}
\newcommand{\Lie}{{\on{Lie}}}
\nc{\Op}{{\on{Op}}}

\nc{\res}{{\on{res}}}

\nc{\tr}{{\on{tr}}}

\nc{\GSp}{{\on{GSp}}} \nc{\GU}{{\on{GU}}} \nc{\SL}{{\on{SL}}}
\nc{\SU}{{\on{SU}}} \nc{\SO}{{\on{SO}}}

\newcommand{\ars}{{\mathfrak a^{rs}}}
\newcommand{\grs}{{\mathfrak g_1^{rs}}}

\nc{\nh}{{\Loc_{J^p}(\tau')}}
\nc{\bnh}{{\Loc_{\breve J^p}(\tau')}}

\nc{\bU}{{\overline{U}}} 
\nc{\IC}{{\on{IC}}}

\newcommand{\Jac}{\mathrm{Jac}}

\newcommand{\br}{\begin{rouge}}
\newcommand{\er}{\end{rouge}}

\newcommand{\bb}{\begin{bluet}}
\newcommand{\eb}{\end{bluet}}

\newcommand{\p}{\perp}

\nc{\ot}{\otimes}

\nc{\oh}{{\operatorname{H}}}
\nc{\gr}{{\operatorname{gr}}}
\nc{\rk}{{\operatorname{rank}}}
\nc{\codim}{{\operatorname{codim}}}
\nc{\img}{{\operatorname{Im}}}
\nc{\Span}{{\operatorname{Span}}}
\nc{\Img}{\operatorname{Im}}

\newcommand{\beqn}{\begin{equation*}}
\newcommand{\eeqn}{\end{equation*}}

\newcommand{\beq}{\begin{equation}}
\newcommand{\eeq}{\end{equation}}

\newcommand{\bern}{\begin{eqnarray*}}
\newcommand{\eern}{\end{eqnarray*}}

\newcommand{\inv}{{\mathbin{/\mkern-4mu/}}}

\usepackage{etoolbox}
\makeatletter
\patchcmd{\@maketitle}
  {\ifx\@empty\@dedicatory}
  {\ifx\@empty\@date \else {\vskip1ex\footnotesize \centering\@date\par\vskip1ex}\fi
   \ifx\@empty\@dedicatory}
  {}{}
\patchcmd{\@adminfootnotes}
  {\ifx\@empty\@date\else \@footnotetext{\@setdate}\fi}
  {}{}{}
\makeatother

\begin{document}
\title[Springer correspondence and 
cohomology of Fano varieties]{Springer correspondence, hyperelliptic curves, and 
cohomology of Fano varieties
}
        \author{Tsao-Hsien Chen}
        \address{Department of Mathematics, University of Chicago, Chicago, 60637, USA
         }
        \email{chenth@math.uchicago.edu}
        \thanks{Tsao-Hsien Chen was partially supported by NSF grant DMS-1702337 and DMS-2001257.}
        \author{Kari Vilonen}
        \address{School of Mathematics and Statistics, University of Melbourne, Australia and Department of Mathematics and Statistics, University of Helsinki, Finland}
         \email{kari.vilonen@unimelb.edu.au}
         \thanks{Kari Vilonen was supported in part by the ARC grants DP150103525 and DP180101445,  the Academy of Finland, the Humboldt Foundation,  the Simons Foundation, and the NSF grant DMS-1402928.}
         \author{Ting Xue}
         \address{ School of Mathematics and Statistics, University of Melbourne, Australia and Department of Mathematics and Statistics, University of Helsinki, Finland}
         \email{ting.xue@unimelb.edu.au}
\thanks{Ting Xue was supported in part by the ARC grants DP150103525 and DE160100975 and the Academy of Finland.}

\begin{abstract}
In~\cite{CVX3}, we have established a Springer theory  for the symmetric pair 
$(\on{SL}(N),\on{SO}(N))$. In this setting  we obtain representations of (the Tits extension) of the braid group rather than just Weyl group representations. These representations arise from cohomology of families of certain (Hessenberg) varieties. In this paper we determine the Springer correspondence explicitly for IC sheaves supported on order 2 nilpotent orbits.  In this process we encounter universal families of hyperelliptic curves. As an application we calculate  the cohomolgy of Fano varieties of $k$-planes in the smooth intersection of two quadrics in an even dimensional projective space.
\end{abstract}
\maketitle

\setcounter{tocdepth}{1} \tableofcontents
\section{Introduction}

In~\cite{CVX3} we have established a Springer correspondence for the symmetric pairs \linebreak$(\on{SL}(N), \on{SO}(N))$. More precisely, let $G=\on{SL}(N)$, $\Lg=\on{Lie}G=\mathfrak{sl}(N)$, and $\theta$ an involution of $G$ such that $K=G^\theta=\on{SO}(N)$.  The involution $\theta$ gives rise to a decomposition $\fg=\fg_0\oplus\fg_1$,  where $d\theta|_{\fg_i}=(-1)^i$, $i=0,1$. Let $\cN$ be the nilpotent cone of $\Lg$ and let $\cN_1=\cN\cap\fg_1$. In~\cite{CVX3} we determine the set of Fourier transforms of all irreducible $K$-equivariant perverse sheaves on $\cN_1$, i.e., the set of Fourier transforms $\fF(\on{IC}(\cO,\cE))$, where $\cO$ is a nilpotent $K$-orbit in $\cN_1$, $\cE$ is an irreducible $K$-equivariant local system on $\cO$, and $\fF:\on{Perv}_K(\Lg_1)\to\on{Perv}_K(\Lg_1)$ is the Fourier transform functor (where we have identified $\Lg_1$ with $\Lg_1^*$ via the Killing form). 
Several things remain to be done. 
For example, we do not know the Fourier transform of a specific $\on{IC}(\cO,\cE)$ in general, although we have some partial conjectures~\cite[Conjecture 7.1]{CVX3}. 
Also, very little is known about 
the geometric properties of the Fourier transforms $\fF(\on{IC}(\cO,\cE))$ (i.e. the character sheaves)
and the work in \emph{loc. cit.} suggests that the geometric properties of $\fF(\on{IC}(\cO,\cE))$ 
are very different from those in the group case.

In this paper we
explore the geometry of the Fourier transforms $\fF(\on{IC}(\cO,\cE))$ 
in the simplest nontrivial case, that is, in the case when $\cO$ is an order 2 nilpotent orbit \footnote{ 
A nonzero nilpotent orbit $\cO$ is called of order 2 if $x^2=0$ for all $x\in\cO$.}.
We show that those Fourier transforms involve cohomology of a beautiful family of non-rational varieties 
and we use the geometry of those varieties
to determine $\fF(\on{IC}(\cO,\cE))$ for 
all order 2 nilpotent orbits $\cO$. 
As an application, we compute the cohomology of Fano varieties of $k$-planes in the smooth intersection of two quadrics in an even dimensional projective space. When the ambient projective space is odd dimensional, the cohomology of such Fano varieties has been computed in~\cite{CVX2} making use of the Springer correspondence in the case of $N$ even.

Let us assume that $N=2n+1$. We show that for trivial local systems on order 2 nilpotent orbits the Fourier transforms are IC sheaves arising from representations of the Tits group $A[2]\rtimes S_N$, which do not factor through $S_N$; see Theorem~\ref{matching}. Here $A[2]$ is the set of order $2$ elements in a fixed maximal $\theta$-split torus $A$ of $G$. In the case of non-trivial local systems the Fourier transforms are IC sheaves arising from representations of the braid group $B_N$ on the cohomology of  a universal family of hyperelliptic curves; see Theorem~\ref{non-trivial even}.

As an application we obtain an explicit formula for the cohomology of the Fano varieties $\on{Fano}_{i}^{2n}$ of $i$-planes in the smooth complete intersection of two quadrics in an even dimensional projective space $\bP^{2n}$. This is an example of the following general strategy (see also \cite{CVX2}). 
Consider pairs of maps $(\pi,\check \pi)$ and diagrams as follows: 
\beq
\label{advanced springer diagram}
\begin{CD}
\mathcal{X}  @>>> \check{\mathcal{X}}
\\
@V{\pi}VV  @VV{\check\pi}V
\\
\cN_1  @>>>  \fg_1
\end{CD}
\eeq
where $\mathcal{X}$ and  $\check{\mathcal{X}}$ are certain families of Hessenberg varieties, see~\cite{CVX1}. In particular, $\mathcal{X}$ and  $\check{\mathcal{X}}$ are orthogonal vector bundles over a partial flag manifold $\cP$ of $K$ in the trivial bundle $\cP\times\fg_1$ so that, by functoriality of the Fourier transform, we have 
\beq
\label{basic decomposition}
\fF(\pi_*\bC_{\mathcal{X}}) \cong \check\pi_*\bC_{\check{\mathcal{X}}}\,\ \ \text{(up to shift)}.
\eeq  
In general the fibers of $\check{\pi}$ are more complicated than those of $\pi$. We make use of  the following strategy:
\beq
\begin{gathered}
\text{The computation of the cohomology of the general fiber of $\check{\mathcal{X}}$ can be}
\\
\text{reduced, via the Fourier transform, to the analysis of the boundary family  ${\mathcal{X}}$.}
\end{gathered}
\eeq
For this to work, we need to know the Fourier transforms of the IC sheaves  that occur in the decomposition of $\pi_*\bC_{\mathcal{X}}$. 
In our setting
we observe that the Fano varieties  
$\on{Fano}_{i}^{2n}$ appear as general fibers of a certain  
families of Hessenberg varieties $\check{\mathcal{X}}$ and   
by applying the principle above we obtain the following theorem. Let us write $${g_{k,m}(q)=\frac{\prod_{l=m-k+1}^m(1-q^l)}{\prod_{l=1}^k(1-q^l)}}$$ for the Poincare polynomial of the Grassmannian of $k$-dimensional subspaces in $\bC^m$. Then:
 \begin{thm}[Theorem~\ref{thm-fano}]\label{}
We have
\beqn
H^{2k+1}(\on{Fano}_{i-1}^{2n},\bC)=0,
\eeqn
\beqn
H^{2k}(\on{Fano}_{i-1}^{2n},\bC)\cong\bigoplus_{j=0}^i L_j^{\oplus M_i(k,j)},
\eeqn
where $M_i(k,j)$ is the coefficient of $q^{k-j(n-i)}$ in $g_{i-j,2n-i-j}(q)$ and the $L_j$ are vector spaces of dimension ${\binom{2n+1} {j}}$\,.
\end{thm}

\begin{remark}
It follows from the proof that there are natural actions of the 
braid group on the 
cohomology $H^{2k}(\on{Fano}_{i-1}^{2n},\bC)$ and the vector space $L_j$ and the isomorphism in the theorem is compatible with those actions.
\end{remark}
\begin{remark}
The Fano varieties $\on{Fano}_{i}^{2n}$ have a concrete interpretation as moduli space of parabolic vector bundles (with extra structure) on $\bbP^1$ (see \cite{C}). Thus the theorem above provides a method to compute the
cohomology of those moduli spaces.
\end{remark}

To prove our main theorems we  establish  and make use of an interesting isomorphism of the cohomology of the stalks of IC sheaves attached to nilpotent orbits of order 2 and the cohomology of stalks of IC sheaves attached to nilpotent orbits in the symplectic Lie algebra $\mathfrak{sp}(2n)$ in the classical case. This is done in \S\ref{sec-stalks}. The proofs of the results in this section are independent of the other sections and  the results are used in \S\ref{sec orbits of order 2} and in \S\ref{sec-Fano}.

{\bf Acknowledgements.} We thank Carl Mautner for useful discussions at the beginning of the project. We also thank the Max Planck Institute for Mathematics in Bonn and the Mathematical Sciences Research Institute in Berkeley for support, hospitality, and nice research environment. Furthermore KV and TX thank the Research Institute for Mathematical Sciences in Kyoto for support, hospitality, and nice research environment. We also thank Cheng-Chiang Tsai, Zhiwei Yun and Xinwen Zhu for their interest in this work. In addition,  TX would like to thank George Lusztig for suggesting her to study the topics in this paper and for helpful discussions about fake degrees. 
The authors would also like to thank the anonymous referees for helpful comments.
Finally, special thanks are due to Dennis Stanton for suppling a combinatorial argument used in section 6.

\section{Springer correspondence}\label{order 2}
In this section we state our main results on Fourier transforms 
of IC complexes supported on nilpotent orbits of order two.

\subsection{Notations} 
Let $G=\on{SL}(N,\bC)$ and $\theta:G\to G$ an involution such that 
$K=G^\theta=\on{SO}(N,\bC)$.  We also write $(G,K) = (\on{SL}(V),\on{SO}(V,Q))$, where $\dim\,V = N$.  We think of $Q$ concretely as a non-degenerate quadratic form on $V$ and we write $\langle\ ,\rangle_Q$ for the non-degenerate bilinear form on $V$ associated to $Q$. The involution $\theta$ induces a grading on $\Lg=\on{Lie}\,G$, i.e. $\Lg=\Lg_0\oplus\Lg_1$, where $d\theta|_{\Lg_i}=(-1)^i$. If we diagonalize $Q$ then the Cartan involution $\theta$ is given by $g\mapsto (g^t)^{-1}$ and $\Lg_1$ consists of symmetric matrices. 

We write $A$ for a maximal $\theta$-split torus of $G$, where $\theta$-split means that $\theta(t)=t^{-1}$ for $t\in A$. The pair $(G,K)$ is a split pair, i.e., $A$ is a maximal torus of $G$.  We write $\Lg^{rs}$ for the regular semi-simple elements in $\Lg$ and we let $\Lg_1^{rs}=\Lg_1\cap\Lg^{rs}$. Furthermore, 
let $\fa=\on{Lie}A$ and $\fa^{rs}=\fa\cap\Lg^{rs}$. Then $\fa$ is a maximal abelian subspace in $\Lg_1$. 

Let us write $\cN$ for the nilpotent cone of $\Lg$ and we let $\cN_1=\cN\cap\Lg_1$ stand for  the nilpotent cone of $\Lg_1$. 
Assume that $N=2n+1$. The $K$-orbits on $\cN_1$ are parametrized by partitions $\lambda$ of $2n+1$ and 
we write $\mO_\lambda$ for the $K$-orbit corresponding to $\lambda$ (see \cite{S}). For $i\in[0,n]$, let $2^i1^{2n+1-2i}$ denote the partition where the part $2$ has multiplicity $i$, and the part $1$ has multiplicity  $2n+1-2i$.
The order $2$ nilpotent orbits are $\cO_{2^i1^{2n+1-2i}}$, $i\in[1,n]$.  For $x\in\cO_{2^i1^{2n+1-2i}}$, $i\in[1,n]$, we have that $A_K(x):=Z_K(x)/Z_K(x)^0=\bZ/2\bZ$. Thus each orbit $\cO_{2^i1^{2n+1-2i}}$ carries a unique non-trivial irreducible $K$-equivariant local system, $i\in[1,n]$; we denote it by $\cE_i$.

By~\cite[Corollary 4.9]{CVX3}, we know that both $\fF(\on{IC}(\cO_{2^i1^{2n+1-2i}},\bC))$ and $\fF(\on{IC}(\cO_{2^i1^{2n+1-2i}},\cE_i))$ are of the form $\on{IC}(\Lg_1^{rs},-)$, where $-$ is an irreducible $K$-equivariant local system on $\Lg_1^{rs}$ corresponding to an irreducible representation of $\pi_1^K(\Lg_1^{rs})=\pi_1^K(\Lg_1^{rs},a)$, $a\in\fa^{rs}$, the equivariant fundamental group.

\subsection{The equivariant fundamental group}Let us describe the equivariant fundamental group $\pi_1^K(\Lg_1^{rs},a)$ concretely. In this subsection $N$ is arbitrary. 

Let $\rW_1=N_K(A)/Z_K(A)$ be the `baby Weyl group'. Recall that we have
\begin{lemma}[\cite{KR}]\label{lemma-involution}
 The natural inclusion $\fa\hookrightarrow\fg_1$ induces 
an isomorphism $k[\fa]^{\rW_1}=k[\fg_1]^K$.
We have the 
relative Chevalley map $f:\fg_1\ra\fg_1\inv K\is\fa\slash \rW_1:=\fc$. 
\end{lemma}
Let $W$ be the Weyl group of $G$. In our setting, we have $\rW_1=W=S_N$. Let us write  
$c\in\fc^{rs}$ for the  image of $a$ under $\ars\ra \fc^{rs}=\ars/W$.  Now,  $B_W:=\pi_1(\fc^{rs},c)=B_N$ is the classical braid group. Then the relative Chevalley map $f:\fg_1\ra \fc$  gives rise to the following commutative diagram with exact rows:
\beq
\begin{CD}
0 @>>> Z_K(a)@>>>\pi_1^K(\fg_1^{rs},a)@>>> B_W@>>> 0
\\
@. @| @VVV @VVV @.
\\
0 @>>> Z_K(a)@>>>\widetilde W=N_K(A)@>>> W@>>> 0\,.
\end{CD}
\eeq
Note that $\widetilde W$ is  the Tits group.
The group $B_W$ acts on $Z_K(a)$ through the quotient $B_W\ra W$. We have that  $Z_K(a) \is A[2]$, the group of order 2 elements in the split torus $A$.
Choosing a section we can split the short exact sequence in the first row and obtain 
\[\pi_1^K(\fg_1^{rs},a)\is A[2]\rtimes B_W.\]
However, note that  the second row in the diagram is not split in general.
In the case when $N$ is odd the lower exact sequence also splits and we have 
\[\widetilde W=N_K(A)\is A[2]\rtimes S_N\,.\]

For the rest of the paper, we will assume $N$ is odd.

\subsection{The local systems $\calL_i$ and $\cF_i$ and the statements of theorems}\label{sec-the local systems}
The local systems on $\fg_1^{rs}$ which are obtained from the Fourier transforms of $\IC(\mO_{2^i1^{2n-2i+1}},\bC)$ turn out to be representations of the Tits group $\widetilde W$. They are the irreducible representations which occur in $\bC[A[2]]$ viewed as a representation of $\widetilde W$. First, let us write $L_\chi$ for the representation of $A[2]$ associated to the character $\chi\in A[2]^\vee$ and so we have 
\beqn
\bC[A[2]] \ = \ \bigoplus_{\chi\in A[2]^\vee} L_\chi
\eeqn
The space $A[2]^\vee$ breaks into orbits $\{\Lambda_i\}_{i=0,...,n}$ under the action of $W=S_N$, where we have numbered the orbits so that $|\Lambda_i| = {\binom{2n+1} {i}}$. Note that here it is crucial that $N=2n+1$ is odd. Concretely, the $\Lambda_i$ consists of characters that attain the value $-1\in\mathbb{G}_m$ precisely $i$ times if $i$ is even and $2n+1-i$ times if $i$ is odd\footnote{We have chosen this parametrization to get a clean statement for our theorem below.}. Thus, for each ${i=0,...,n}$ we obtain an irreducible representation of $\widetilde W$ as follows
\beq\label{L_i}
L_i \ = \  \bigoplus_{\chi\in\Lambda_i} L_\chi\,.
\eeq
We write $\calL_i$ for the irreducible $K$-equivariant local system on $\fg_1^{rs}$ corresponding to the representation $L_i$ via the map $\pi_1^K(\fg_1^{rs},a)\to\widetilde W$.

\begin{thm}\label{matching}
We have $\mathfrak F(\IC(\mO_{2^i1^{2n-2i+1}},\bC))=\IC(\fg_1,\mL_i)$, $i=0,\ldots,n$.
\end{thm}

We will now define irreducible $K$-equivariant local systems $\cF_i$ on $\fg_1^{rs}$ which are obtained from the Fourier transforms of $\IC(\mO_{2^i1^{2n-2i+1}},\cE_i)$'s. These local systems arise as representations of $\pi_1(\fc^{rs},a) = B_N$. To this end we consider the following universal family $C$ of hyperelliptic curves of genus $n$:
\beqn
\begin{gathered}
\text{To each $a=(a_1, \dots,a_{2n+1})\in \fa^{rs}$ we associate the hyperelliptic}
\\
\text{curve $C_a$ which ramifies at $\{a_1, \dots,a_{2n+1}, \infty\}$}\,.
\end{gathered}
\eeqn
(Here we have chosen a Cartan subspace $\fa$ of $\Lg_1$ such that it consists of diagonal matrices.) 
This family gives us a monodromy representation $B_N \to Sp(H^1(C_a,\bC))$. Note that, by \cite{A} (see also \cite{KS}) this monodromy representation has a Zariski dense image, in particular, the monodromy is infinite. 
From this we get a monodromy representation on the Jacobian of $C_a$ which we break into primitive parts:
\beqn
B_N \to H^i(\Jac(C_a),\bC)_{prim}\is(\wedge^iH^1(C_a,\bC))_{prim}\ \ i\in[1,n]\,.
\eeqn
Associated to this representation we obtain a local system $\cF_i$ on $\fg_1^{rs}$. Note that the part $A[2]$ of $\pi_1^K(\Lg_1^{rs})$ acts trivially on $\cF_i$.

\begin{thm}\label{non-trivial even} We have
$\mathfrak F(\IC({\cO_{2^{i}1^{2n+1-2i}}},\cE_{i}))=\IC(\fg_1,\mF_{i})$, $i=1,\ldots,n$.
\end{thm}

\section{Bijection of sets}\label{sec-order2}
In this section we show that
\begin{prop}\label{FT of order 2 trivial}
There is a permutation $s$ of the set $\{0,...,n\}$ such that 
\[\mathfrak F(\IC({\cO}_{2^k1^{2n-2k+1}},\bC))\is \IC(\fg_1,\mL_{s(k)})\]
for $k\in\{0,...,n\}$.
\end{prop}

\begin{prop}\label{s}
There is a permutation $s\in S_n$ such that 
\[\mathfrak F(\IC(\mO_{2^{i}1^{2n-2i+1}},\mE_{i}))\is 
\IC(\fg_1,\mathcal F_{s(i)})
,\ i\in[1,n].\]
\end{prop}

Let us denote by $\on{OGr}(s,2n+1)$ the variety of $s$-dimensional isotropic subspaces in $\bC^{2n+1}$ with respect to a non-degenerate bilinear form. We write $V_i$ for a vector subspace of $V=\bC^{2n+1}$ of dimension $i$ and $V_i^\p=\{x\in V\,|\,\langle x,V_i\rangle_Q=0\}$.

\subsection{Proof of Proposition \ref{FT of order 2 trivial}}Let $P$ be the parabolic subgroup in $G$ that stabilizes the partial flag $0\subset V_n^0\subset V_n^{0\perp}\subset V$, where $V_n^0=\text{span}\{e_{1},\ldots,e_{n}\}$ (here $(e_i)$ is a basis  of $V$ such that $\langle e_i,e_j\rangle_Q=\delta_{i+j,2n+2}$). Let 
$P_K=K\cap P$ and
\beqn
E=\{x\in\Lg_1\,|\,xV_n^{0\perp}=0\}.
\eeqn
By Reeder \cite{R}, we have the following resolution of singularities of $\bar{\cO}_{2^n1}$
\beq\label{Redder's map}
{\pi_{2^n1}}:\xymatrix{
K\times^{P_K}E\cong\{(x,0\subset V_n\subset V_n^\p\subset V)\,|\,x\in\Lg_1,\,xV_n^\p=0\}\ar[r]&\bar{\cO}_{2^n1}.}
\eeq
\begin{lemma}
The map $\pi_{2^n1}$ is semi-small and we have
\beq\label{decomp of 2^n1}
\pi_{2^n1*}\bC[-]\cong\bigoplus_{k=0}^n\IC({\cO}_{2^k1^{2n-2k+1}},\bC).
\eeq
\end{lemma}
\begin{proof}
Let $x\in\cO_{2^k1^{2n-2k+1}}$. Suppose that $V_n\in\pi_{2^n1}^{-1}(x)$. Then $\text{Im}\,x\subset V_n$ and $\dim \text{Im}\,x=k$. We have that
\beqn
\begin{gathered}
\pi_{2^n1}^{-1}(x)\cong\{n-k\text{ dimensional isotropic subspaces of }(\on{Im}\,x)^\perp/\on{Im}\,x\}\\\hspace{2in}\cong\on{OGr}(n-k,2n-2k+1).
\end{gathered}
\eeqn
Thus $\pi_{2^n1}^{-1}(x)$ is irreducible and it is easy to check that 
$
2\dim\pi_{2^n1}^{-1}(x)=(n-k)(n-k+1)=\on{codim}_{\bar{\cO}_{2^n1}}\cO_{2^k1^{2n-2k+1}}.
$ In the last equation we use the fact that $\on{dim}\cO_{2^k1^{2n-2k+1}}=k(2n+1-k)$. The lemma follows from the decomposition theorem.
\end{proof}
Consider the map
\beq\label{dual resolution of 2^n}
{\check{\pi}_{2^n1}}:\xymatrix{K\times^{P_K}E^\perp\cong\{(x,0\subset V_n\subset V_n^\p\subset V)\,|\,x\in\Lg_1,\,xV_n\subset V_n^\p\}\ar[r]&\Lg_1,}
\eeq
where $E^\perp$ denotes the orthogonal complement of $E$ in $\Lg_1$ with respect to the non-degenerate form on $\Lg_1$ (given by the restriction of the Killing form on $\Lg$). 
For $x\in\Lg_1$, we have
\beqn
\check{\pi}_{2^n1}^{-1}(x)=\{V_n\text{ maximal isotropic subspace in }V\,|\,x V_n\subset V_n^\perp\}.
\eeqn
For $x\in\fg_1^{rs}$ 
there are $2^{2n}$ such maximal isotropic subspaces in $V$ and 
the centralizer $Z_K(x)$ acts simply-transitively on those subspaces (see  \cite{Re,BG}).

Consider the  $K$-equivariant local system $\mL=(\check\pi_{2^n1})_*\bC|_{\fg_1^{rs}}$  of rank 
$2^{2n}$ on $\fg_1^{rs}$. Our first goal is to describe this local system.
Fix $a\in\ars$. The stalk $L:=\mL_a$ carries an action of the $K$-equivariant fundamental group 
$\pi_1^K(\fg_1^{rs},a)=Z_K(a)\rtimes B_N$.  As $Z_K(a)$ acts simply transitively on  $\check{\pi}_{2^n1}^{-1}(x)$, as was remarked above, we can identify $L$ with $\bC[Z_K(a)]= \bC[A[2]]$. Furthermore, the action of $\pi_1^K(\fg_1^{rs},a)$ factors through the Tits group $\widetilde W$ and it coincides with the canonical representation of $\widetilde W$ on $\bC[A[2]]$. Let us recall the irreducible $K$-equivariant local systems $\calL_i$ from~\S\ref{order 2} and then 
\beq\label{mL_i}
\mL=\bigoplus_{i=0}^n\mL_i.
\eeq

\begin{lemma}\label{lem-decomposition pi{2n1}}
We have 
\beq\label{decomposition cpi{2n1}}
(\check\pi_{2^n1})_*\bC[-]=\IC(\Lg_1,\mL)=\bigoplus_{i=0}^n \IC(\Lg_1,\mL_i).
\eeq
In particular, the map $\check\pi_{2^n1}$ is small.
\end{lemma} 
\begin{proof}
By the decomposition theorem, $\bigoplus_{i=0}^n\IC(\Lg_1,\mL_i)$ is a 
direct summand of $(\check\pi_{2^n1})_*\bC[-]$. On the other hand, since, by the functoriality of the Fourier transform,  
\beq\label{eqn-ftof pi}
\mathfrak F((\pi_{2^n1})_*\bC[-])\is(\check\pi_{2^n1})_*\bC[-],
\eeq the decomposition in 
(\ref{decomp of 2^n1}) implies that $(\check\pi_{2^n1})_*\bC[-]$ has exactly $n+1$ irreducible summands.
The lemma follows from \eqref{mL_i}.
\end{proof}

Proposition \ref{FT of order 2 trivial} follows from \eqref{decomp of 2^n1}, \eqref{decomposition cpi{2n1}} and \eqref{eqn-ftof pi}.

\subsection{Proof of Proposition \ref{s}}\label{non-trivial on 2^n}

Let $V=\bC^{2n+1}$ and $W=V\oplus\bC$. 
For any $x\in\fg_1$, consider the following two quadrics 
in $\mathbb P(W)$: \[\tQ(v,a)=\langle v,v\rangle_Q=0\ \text{and}\ \
\tQ_x(v, a)=\langle v,xv\rangle_Q+a^2=0.\]
Define
\beq\label{Def of F}
F=\{(x,W_n\subset W)\,|\, 
x\in\fg_1,\ \dim W_n=n,\ \ 
\mathbb P(W_n)\subset\tQ\cap\tQ_x\}.
\eeq
The map $W\ra W$ given by $(v, a)\mapsto (v,-a)$ induces 
an involution $\sigma$ on $F$ and the set $F^\sigma$ of fixed points  
is equal to \[F^\sigma=\{(x,V_n\subset V)\,|\,x\in\fg_1,\ V_n\text{ maximal isotropic},\,x V_n\subset V_n^\perp\}\cong K\times^{P_K}E^\perp\text{ (see \eqref{dual resolution of 2^n})}.\]

Note that for $x\in\fg_1^{rs}$, the pencil of quadrics spanned by 
$\tQ$ and $\tQ_x$ is non-degenerate and 
contains exactly $2n+2$ singular elements, namely, the 
quadric $\tQ$ at infinity and the $2n+1$ quadrics 
$\lambda_i\tQ-\tQ_x$, where $\lambda_1,...,\lambda_{2n+1}$
are the roots of $p(t)=\det (t\cdot\id-x)$. 

We denote by $\pi:F\ra\fg_1$ the natural projection. 
The fiber $F_x=\pi^{-1}(x)$ of $\pi$ over $x\in\fg_1^{rs}$
is the Fano variety of $(n-1)$-dimensional subspaces contained in the smooth 
complete intersection $\tQ\cap\tQ_x$. According to \cite{Re}, $F_x$ is a torsor over $\Jac(C_x)$, 
where $C_x$ is the smooth projective hyperelliptic curve with affine equation:
\[y^2=\prod_{i=1}^{2n+1}(t-\lambda_i).\]
Moreover, the action of the involution $\sigma$ on $F_x$ is compatible with 
the inversion on $\Jac(C_x)$. In particular, the set $F^\sigma_x$ of fixed points is a 
$\Jac(C_x)[2]$-torsor, where $\Jac(C_x)[2]$ consists of 2-torsion points of  $\Jac(C_x)$.

The discussion above has the following relative version. Namely,
let $\pi^C:C\ra\fg_1^{rs}$ be the family of curves
$C_x$ and let $\Jac(C)\ra\grs$ denote the corresponding relative Jacobian.
Let $F|_{\grs}\ra\grs$ be the family of Fano varieties of 
$(n-1)$-dimensional subspaces contained in the smooth
complete intersection $\tilde Q\cap\tilde Q_x$.
Then $\Jac(C)$ acts naturally on $F|_{\grs}$, and  
$F|_{\grs}$ is a $\Jac(C)$-torsor under this action. Similarly, 
$\Jac(C)[2]$ acts on $F^\sigma$, and 
$F^\sigma$ is a $\Jac(C)[2]$-torsor under this action.
We have the following observation.

\begin{lemma}
(1)
The action map $\Jac(C)\times_{\grs}F^{\sigma}|_{\grs}\ra F|_{\grs}$
factors through an
isomorphism 
\[(\Jac(C)\times_{\grs}F^{\sigma}|_{\grs})/\Jac(C)[2]\is F|_{\grs}.\] 
Here $\Jac(C)[2]$ acts on $\Jac(C)\times_{\grs}F^{\sigma}|_{\grs}$ via the diagonal action.
\

(2) For any $a\in\grs$, 
there is a canonical isomorphism 
\[H^i(\Jac(C_a),\bC)
\is H^i((\Jac(C_a)\times F_a^\sigma)/\Jac(C_a)[2],\bC)\]
compatible with the monodromy actions of $\pi_1^K(\grs,a)$ on both sides.
\end{lemma}
\begin{proof}
(1) is clear. For (2), we observe that, by K\"unneth formula, we have 
\beq\label{Kunneth}
H^i((\Jac(C_a)\times F_a^\sigma)/\Jac(C_a)[2],\bC)\is (H^i(\Jac(C_a),\bC)
\otimes H^0(F_a^\sigma,\bC))^{\Jac(C_a)[2]}
\eeq
where the right hand side is the  
$(\Jac(C_a)[2])$-fixed vectors in $H^i(\Jac(C_a),\bC)
\otimes H^0(F_a^\sigma,\bC)$. Since the action of 
$\Jac(C_a)[2]$ on 
$H^i(\Jac(C_a),\bC)$ is trivial\footnote{To see this, we observe that the action of 
$\Jac(C_a)[2]$ on $H^i(\Jac(C_a))$ is the restriction of the action of  
$\Jac(C_a)$ on $H^i(\Jac(C_a))$. Since $\Jac(C_a)$ is connected the latter action  is trivial.} we have 
\beq\label{trivial action}
(H^i(\Jac(C_a),\bC)
\otimes H^0(F_a^\sigma,\bC))^{\Jac(C_a)[2]}\is H^i(\Jac(C_a),\bC),
\eeq
Combining (\ref{Kunneth}) and (\ref{trivial action}),
we obtain a canonical isomorphism 
\beq\label{can iso}
H^i((\Jac(C_a)\times F_a^\sigma)/\Jac(C_a)[2],\bC)\is H^i(\Jac(C_a),\bC).
\eeq
Since the isomorphisms in (\ref{Kunneth}), (\ref{trivial action}) are compatible with the 
monodromy actions, so is the composition in (\ref{can iso}). Thus (2) follows. 
\end{proof}
It follows immediately that
\begin{corollary}\label{coro-mono}
There is a canonical isomorphism $H^i(F_a,\bC)\is H^i(\Jac(C_a),\bC)$
compatible with the monodromy actions of $\pi_1^K(\grs,a)$.
\end{corollary}

We use the corollary above to study the monodormy of the 
family of Fano varieties $F|_{\grs}\ra\grs$.
To begin with, we observe that over the Kostant section $\kappa:\fc^{rs}\hookrightarrow\fg_1^{rs}$, the family 
$\pi^C:C\ra\grs$
is the universal family of hyperelliptic curves
of genus $n$. As mentioned before, the monodromy representation of $\pi_1^K(\grs,a)$ on $H^1(C_a,\bC)$ 
is irreducible and the image of $\pi_1^K(\grs,a)\to Sp(H^1(C_a,\bC))$ is Zariski dense.
This fact together with the corollary above implies that the monodromy representation of $\pi_1^K(\fg_1^{rs},a)$ on 
\[
H^i(F_a,\bC)_{prim}\is H^i(\Jac(C_a),\bC)_{prim}\is
(\wedge^iH^1(C_a,\bC))_{prim}
\]
is irreducible for $i=1,...,n$, moreover, the corresponding monodromy group is \emph{infinite}.
Thus the corresponding irreducible $K$-equivariant local systems on $\fg_1^{rs}$ are
\beq\label{eqn-def Fi}
(\on{R}^i\pi_*\bC|_{\fg_1^{rs}})_{prim}=\cF_i
\eeq
where $\cF_i$ is the local system defined in \S\ref{sec-the local systems}.
We have \[\dim\mF_i=\binom {2n}{i}-\binom{2n}{ i-2}\] and 
$\mathcal F_i\ncong\mathcal F_j$ for $i\neq j$. 

Recall the local systems $\{\mL_i\}_{i=0,...,n}$ defined in \S\ref{sec-the local systems}.
Since each $\mL_i$
has
finite monodromy, we have 
$\mathcal F_i\ncong\mL_j$ for all $i,j$.
As there is only one nontrivial irreducible $K$-equivariant  local system $\mE_i$ on 
each
$\mO_{2^{i}1^{2n-2i+1}}$ ($i\geq 1$), 
Proposition \ref{s} follows from Proposition \ref{FT of order 2 trivial} and the following proposition
\begin{prop}\label{supp} For $i=1,...,n$,
$\mathfrak F(\IC(\fg_1,\mathcal F_i))$ is supported on 
$\bar{\mathcal O}_{2^n1}$.
\end{prop}

\begin{proof}
To begin with, we show that the variety $F$ defined in~\eqref{Def of F} is isomorphic to 
\beq\label{equiv of F}
\begin{gathered}
\{(x,V_n,\alpha)\,|\,x\in\Lg_1,\,V_n\text{ maximal isotropic in }V,\,\alpha\in V_n^*,\\\langle v,xv'\rangle+\alpha(v)\alpha(v')=0\text{ for all }v,v'\in V_n\}.
\end{gathered}
\eeq
Consider the projection maps $p_1:W=V\oplus\bC\to V$ and $p_2:W=V\oplus\bC\to\bC$. Let $(x,W_n)\in F$ and let $V_n=p_1(W_n)$. It is easy to see that $p_1|_{W_n}:W_n\to V_n$ is an isomorphism. Moreover $V_n\subset V_n^\p\subset V$. Let $\alpha:V_n\to\bC$ be defined by $\alpha(v)=p_2\circ (p_1|_{W_n})^{-1}v$. One checks readily that $(x,V_n,\alpha)$ belongs to the set in~\eqref{equiv of F}. Conversely, given $(x,V_n,\alpha)$, we let $W_n=\{(v,\alpha(v))\,|\,v\in V_n\}$. Then $(x,W_n)\in F$.  It follows that $F$ is smooth.

 Since $\pi:F\ra\fg_1$ is proper, the decomposition theorem and~\eqref{eqn-def Fi} imply that 
$\IC(\fg_1,\mF_i)$ (up to shift) is a summand of $\pi_*\bC_F$, $i=1,\ldots,n$.  Thus it is enough to
show that $\mathfrak F(\pi_*\bC_F)$ is supported on $\bar\mO_{2^n1}$.

Recall the maps
$\pi_{2^n1}:K\times^{P_K}E\ra\bar\mO_{2^n1}$ 
and 
$\check\pi_{2^n1}:K\times^{P_K}E^\p$ defined 
in (\ref{Redder's map}) and (\ref{dual resolution of 2^n}). Let $q:
K/P_K\times \fg_1\is K\times^{P_K}\fg_1\ra
K\times^{P_K}(\fg_1/E^\bot)$ denote the quotient map. Note that the non-degenerate invariant form on $\Lg_1$ induces isomorphisms 
$E\is(\fg_1/E^\bot)^*$, $\fg_1\is\fg^*_1$, and 
under these isomorphisms,
the dual map of 
$q$  can be identified with the natural embedding 
\[\check q:K\times^{P_K}E\ra K\times^{P_K}\fg_1.\]
Let 
\beqn
\begin{gathered}
\bar F=\{(\bar x,V_n,\alpha)\,|\,\bar x\in\Lg_1/\{x\in\Lg_1\,|\,xV_n\subset V_n^\p\},\,V_n\text{ maximal isotropic in }V,\,\alpha\in V_n^*,\\\qquad\qquad\langle v,xv'\rangle+\alpha(v)\alpha(v')=0\text{ for all }v,v'\in V_n\text{ and $x\in \Lg_1$ any lift of $\bar x$}\}.
\end{gathered}
\eeqn
We have the following Cartesian diagram:
\[\xymatrix{
F\ar[r]^{\bar q}\ar[d]^u&\bar F\ar[d]^{\bar u}\\
 K\times^{P_K}\fg_1\ar[r]^-{q}&K\times^{P_K}(\fg_1/E^\p)\,,}\]
where $u: (x,V_n,\alpha)\mapsto(V_n,x)$, $\bar q:(x,V_n,\alpha)\mapsto(\bar x,V_n,\alpha)$ and $\bar u:(\bar x,V_n,\alpha)\mapsto(V_n,\bar x)$.

Let us  write 
$
pr: K\times^{P_K}\fg_1\to\fg_1
$ for the natural projection. We have that, by functoriality of the Fourier transform,
\beqn
\begin{gathered}
\fF(\pi_*\bC_F)=\fF(pr_*u_*\bC_F)=\fF(pr_*u_*\bar q^*\bC_{\bar F})=\fF(pr_*q^*\bar u_*\bC_{\bar F})\\
=pr_*\check q_*\fF(\bar u_*\bC_{\bar F})=\pi_{2^n1*}\fF(\bar u_*\bC_{\bar F}).
\end{gathered}
\eeqn
It follows that $\fF(\pi_*\bC_F)$ is supported on $\bar\cO_{2^n1}$.
\end{proof}

\section{Proof of the main theorems}\label{sec orbits of order 2}
In this section we prove Theorem~\ref{matching} below, but give a proof of Theorem~\ref{non-trivial even} only for  orbits $\cO_{2^i1^{2n+1-2i}}$ when $i$ is even. We defer the case of odd $i$ to~\cite[Proposition 5.4]{CVX1}.

\subsection{Proof of Theorem \ref{non-trivial even} for $i$ even}\label{matching for 
non trivial 2^n}

Consider the maps
$$
\upsilon:\{(x,0\subset V_{n-1}\subset V_{n}\subset V_n^\p\subset V_{n-1}^\p\subset V=\bC^{2n+1})\,|\,x\in\Lg_1,\,xV_n^\p=0, x\,V_{n-1}^\p\subset V_{n-1}\}:=Y$$$$\to\cN_1,$$
\beqn
\check{\upsilon}:\{(x,0\subset V_{n-1}\subset V_{n}\subset V_n^\p\subset V_{n-1}^\p\subset V=\bC^{2n+1})\,|\,x\in\Lg_1,\,xV_{n-1}\subset V_{n}^\p\}\to\Lg_1.
\eeqn
We have that 
\beqn
\dim\,Y=n^2+2n-2,\ \ \img\,\upsilon=\bar{\cO}_{2^n1}
\eeqn
and
\beqn
\fF(\upsilon_*\bC[-])=\check{\upsilon}_*\bC[-].
\eeqn
In the following we prove the theorem by studying the decompositions of $\upsilon_*\bC[-]$ and 
$\check\upsilon_*\bC[-]$.

We first study the decomposition of $\check\upsilon_*\bC[-]$.
Since in the decomposition of $\upsilon_*\bC[-]$ only $\IC$ complexes supported on $\cO_{2^i1^{2n+1-2i}}$ appear and the Fourier transform of such complexes all have full support (see Proposition \ref{FT of order 2 trivial} and Proposition \ref{s}), we conclude that all $\IC$ complexes appearing in the decomposition of $ \check{\upsilon}_*\bC[-]$ are supported on all of $\Lg_1$.

Let $F$ be the smooth variety introduced in~\eqref{Def of F}. Recall its equivalent definition in~\eqref{equiv of F} and the involution $\sigma$ on $F$. Let $$Z=\{(x,0\subset V_n\subset V_n^\bot\subset V)\,|\,x\in\fg_1,\on{rank}(\bar x:V_n\ra V/V_n^\bot)\leq 1\}$$ 
and 
$$Z_0=\{(x,0\subset V_n\subset V_n^\bot\subset V)\,|\,x\in\fg_1,\on{rank}(\bar x:V_n\ra V/V_n^\bot)=0\}.$$ We have $\dim Z=\dim\Lg_1+n$ and $\dim Z_0=\dim\Lg_1$. Consider the map
\beq\label{Kummer construction}
f:F\ra Z,\ (x,V_n,\alpha)\mapsto(x,V_n).
\eeq 
Then $f$ is a branched double cover with branch locus 
$Z_0$
\footnote{ In fact, the map $f:F\ra Z$ realizes $Z$ as the (GIT) quotient 
$F/\sigma$.}. Let $Z^0:=Z-Z_0$. Then 
$f^0:=f|_{Z^0}:F\times_ZZ^0\ra Z^0$ is a double cover and we have
$f^0_*\bC=\bC\oplus\cE$, where $\cE=(f^0_*\bC)^{\sigma=-id}$ is a rank one local system on $Z^0$.
It follows that
\beq\label{decomposition of F}
f_*\bC_F[\dim F]=\IC(Z,\bC)\oplus \IC(Z,\cE).
\eeq
Let $a\in\ars$ and let $pr: Z\to \Lg_1$ be the natural projection. 
By (\ref{decomposition of F}), we have 
\[pr_*\IC(Z,\bC)[-\dim F]|_a\is H^*(F_a,\bC)^{\sigma=id},\ \ 
pr_*\IC(Z,\cE)[-\dim F]|_a\is H^*(F_a,\bC)^{\sigma=-id}.\]
Choosing an isomorphism $\Jac(C_a)\is F_a$, we may identify $\sigma$ with the inversion on 
$\Jac(C_a)$. Therefore $$H^*(F_a,\bC)^{\sigma=id}\is\oplus_{i=2j }\wedge^i H^1(C_a,\bC)[-i],\ 
H^*(F_a,\bC)^{\sigma=-id}\is\oplus_{i=2j+1}\wedge^i H^1(C_a,\bC)[-i].$$
This implies that 
\beq \label{even and odd}
\begin{gathered}
\text{$\IC(\fg_1,\oplus_{i=2j}\oplus_{k=0}^{j}\cF_{2k}[-i])$ (resp. $\IC(\fg_1,\oplus_{i=2j+1 }\oplus_{k=0}^{j}\cF_{2k+1}[-i]$) appears}\\\text{ in 
$pr_*\IC(Z,\bC)$ (resp. $pr_*\IC(Z,\cE)$) as a direct summand (up to shift).}
\end{gathered}
\eeq Moreover,
these are the only $\IC$ complexes with full support appearing in the decomposition.

We have the following factorization of $\check\upsilon$ 
\[
\xymatrix{
\{(x,0\subset V_{n-1}\subset V_{n}\subset V_n^\p\subset V_{n-1}^\p\subset V=\bC^{2n+1})\,|\,xV_{n-1}\subset V_{n}^\p\}\ar[r]^-{\check\upsilon_1}&Z\ar[r]^{\check\upsilon_2}
&\Lg_1,}
\]
where $\check{\upsilon}_1:(x,V_{n-1}\subset V_n)\mapsto(x,V_n)$ and $\check{\upsilon}_2:(x,V_n)\mapsto x$.
Note that $\check\upsilon_2|_{Z_0}:Z_0\ra\fg_1$ is equal to the map $\check{\pi}_{2^n1}$ in (\ref{dual resolution of 2^n}).

The map $\check{\upsilon}_1$
is one-to-one over $Z-Z_0$\footnote{ The inverse is given by 
$(x,V_n)\ra (x,V_{n-1}\subset V_n)$ where $V_{n-1}:=\on{Ker}(\bar x:V_n\ra V/V_n^\bot)$.} and 
is a $\mathbb P^{n-1}$-bundle over $Z_0$. It follows that  
\[\check\upsilon_{1*}\bC[-]=\IC(Z,\bC)\oplus\bigoplus_{a=0}^{n-2}\bC_{Z_0}[-][n-2-2a].\]
Since all $\IC$ complexes appearing in the decomposition of $\check{\upsilon}_*\bC[-]$ are supported on all of $\fg_1$,
Lemma \ref{lem-decomposition pi{2n1}} and
 \eqref{even and odd}
imply that
\beq\label{decomposition-dual}
\check{\upsilon}_*\bC[-]\cong\check\upsilon_{2*}\check\upsilon_{1*}\bC[-]
\cong\bigoplus_{k=0}^{[\frac{n}{2}]}\IC(\Lg_1,\bigoplus_{s=0}^k\cF_{2s})[\pm(n-2k)]\oplus\bigoplus_{a=0}^{n-2}\IC(\fg_1,\bigoplus_{s=0}^ n\mL_s)[n-2-2a],
\eeq
here $\cF_0=\bC$.

We now study the decomposition of $\upsilon_*\bC[-]$.
Our goal is to prove the following 
\beq\label{decomp-1}
\begin{gathered}
\upsilon_*\bC[n^2+2n-2]\cong \bigoplus_{a=0}^{n-2}\bigoplus_{i=0}^n\IC({\cO_{2^{i}1^{2n-2i+1}}},\bC)[n-2-2a]\\
\oplus\bigoplus_{i=1}^{[\frac{n}{2}]}\bigoplus_{a_i=0}^{n-2i}\IC({\cO_{2^{2i}1^{2n+1-4i}}},\cE_{2i})[n-2i-2a_i]\oplus\bigoplus_{a=0}^n\IC(\cO_{1^{2n+1}},\bC)[n-2a].
\end{gathered}
\eeq
 Taking Fourier transform of (\ref{decomposition-dual}) and using Proposition \ref{FT of order 2 trivial}, we see that 
\beq\label{directsummand}
\begin{gathered}
\mathfrak F(\bigoplus_{a=0}^{n-2}\IC(\fg_1,\bigoplus_{s=0}^n\mL_s)[n-2-2a])\is
\bigoplus_{a=0}^{n-2}\bigoplus_{i=0}^n\IC({\mO_{2^{i}1^{2n-2i+1}}},\bC)[n-2-2a]\\\text{ is a direct summand of $\upsilon_*\bC[-]$.}
\end{gathered}
\eeq
 Moreover, these contain all the $\IC(\cO_{2^i1^{2n-2i+1}},\bC)$, $i\geq 1$, that appears in the decomposition of $\upsilon_*\bC[-]$. Now we determine those $\IC(\cO_{2^i1^{2n-2i+1}},\cE_{i})$, $i\geq 1$,  that appears in the decomposition of $\upsilon_*\bC[-]$.

Let $x_i\in\cO_{2^{i}1^{2n-2i+1}}$, $i\in[1,n]$. We have that $\upsilon^{-1}(x_i)$ is a quadric bundle over
$
\upsilon_0^{-1}(x_i)\cong\on{OGr}(n-i,2n-2i+1)$, 
with fibers quadric of the form $\sum_{k=1}^{i}a_k^2=0$ in $\bP^{n-1}=\{[a_1:...:a_n]\}$, and
\beqn
2\dim \upsilon^{-1}(x_i)=\on{codim}_E\cO_{x_i}+n-2=2(n-2)+{(n-i)(n-i+1)}.
\eeqn
Here $\upsilon_0=\pi_{2^n1}$ is the map defined in \eqref{Redder's map}. We have that (see \eqref{decomp of 2^n1})
\beqn\label{decomp}
\upsilon_{0*}\bC[n^2+n]\cong\bigoplus_{i=0}^n\IC({\cO_{2^{i}1^{2n-2i+1}}},\bC),
\eeqn
which implies that
\beq\label{eqn-stalk}
H^{k}(\upsilon_0^{-1}(x_j),\bC)\cong\mH^{k-n^2-n}_{x_j}(\bigoplus_{i=0}^n\IC({\mO_{2^{i}1^{2n-2i+1}}},\bC)).
\eeq
We have $H^{\text{odd}}(\upsilon_0^{-1}(x_j),\bC)=H^{\text{odd}}(\upsilon^{-1}(x_j),\bC)=0$ and $$H^{2k}(\upsilon^{-1}(x_j),\bC)\cong\bigoplus_{a=0}^{n-2}H^{2a}(Q_j,\bC)\otimes H^{2k-2a}(\upsilon_0^{-1}(x_j),\bC),$$ where $Q_j$ is a quadric of the form $\sum_{s=1}^{j}a_s^2=0$ in $\bP^{n-1}=\{[a_1:...:a_n]\}$. Note that $H^{2a}(Q_j,\bC)=\bC$ for $0\leq a\leq n-2$ if $j$ is odd, or if $j$ is even and $2a\neq 2n-j-2$, and for $j$ even, $H^{2n-j-2}(Q_j,\bC)\cong\bC\oplus H_{\operatorname{prim}}^{2n-j-2}(Q_j,\bC)$, where $\dim H_{\operatorname{prim}}^{2n-j-2}(Q_j,\bC)=1$. Thus
\begin{eqnarray}
 &&H^{2k}(\upsilon^{-1}(x_j),\bC)\cong\bigoplus_{a=0}^{n-2}H^{2k-2a}(\upsilon_0^{-1}(x_j),\bC)\text{ if }j\text{ is odd}\label{eqn-cohomology}\\&&H^{2k}(\upsilon^{-1}(x_j),\bC)\cong\bigoplus_{a=0}^{n-2}H^{2k-2a}(\upsilon_0^{-1}(x_j),\bC)\oplus H^{2k-2n+j+2}(\upsilon_0^{-1}(x_j),\bC)\text{ if }j\text{ is even}\label{eqn-cohomology2}.
\end{eqnarray}
It follows from \eqref{eqn-stalk} that
\beq\label{eqn-stalk2}
\mH^{2k-n^2-2n+2}_{x_j}\bigoplus_{a=0}^{n-2}\bigoplus_{i=0}^n\IC({\cO_{2^{i}1^{2n-2i+1}}},\bC)[n-2-2a]\cong\bigoplus_{a=0}^{n-2}H^{2k-2a}(\upsilon_0^{-1}(x_j),\bC).
\eeq
Thus \eqref{eqn-stalk2} together with \eqref{eqn-cohomology} implies that if $j$ is odd, then
\beqn
\mH^{2k}_{x_j}\bigoplus_{a=0}^{n-2}\bigoplus_{i=0}^n\IC({\mO_{2^{i}1^{2n-2i+1}}},\bC)[n-2-2a]\cong\mH^{2k}_{x_j}(\upsilon_*\bC[-]).
\eeqn 
In view of \eqref{directsummand}, we conclude that $\IC(\cO_{2^j1^{2n+1-2j}},\cE_j)$, $j$ odd, does not appear in the decomposition of $\upsilon_*\bC[-]$.

By the discussion above, we can assume that
\beq\label{directsumdec}
\begin{gathered}
\upsilon_*\bC[-]=\bigoplus_{a=0}^{n-2}\bigoplus_{i=0}^n\IC({\cO_{2^{i}1^{2n-2i+1}}},\bC)[n-2-2a]\\
\qquad\oplus\bigoplus_{i=1}^{[\frac{n}{2}]}\bigoplus_{a_i=0}^{k_i}\IC({\cO_{2^{2i}1^{2n+1-4i}}},\cE_{2i}^{\oplus m^i_{a_i}})[k_i-2a_i]\oplus\cdots
\end{gathered}
\eeq
where  the $m^i_{a_i}$'s are to be determined and $\cdots$ is a sum of $\bC_{\{0\}}[-]$.
We show that
\beq\label{directsummand2}
k_i=n-2i,\ m^i_0=m^i_{n-2i}=1.
\eeq
Note that in \eqref{eqn-cohomology2}, $H^{2k-2n+j+2}(\upsilon_0^{-1}(x_j),\bC)\neq 0$ if and only if $0\leq 2k-2n+j+2\leq(n-j)(n-j+1)$, i.e. if and only if 
\beqn
2n-j-2\leq2k\leq(n-j)(n-j+1)+2n-j-2,
\eeqn since $\dim\upsilon_0^{-1}(x_j)=(n-j)(n-j-1)/2$. Thus for all $l>(n-2i)(n-2i+1)+2n-2i-2$, we must have
\beqn
\mH_{x_{2i}}^{l-n^2-2n+2}\bigoplus_{a_i=0}^{k_i}\IC({\cO_{2^{2i}1^{2n+1-4i}}},\cE_{2i}^{\oplus m^i_{a_i}})[k_i-2a_i]=0,
\eeqn
i.e. for all $l> n-2i-\dim\cO_{2^{2i}1^{2n+1-4i}}$, $\mH_{x_{2i}}^{l}\bigoplus_{a_i=0}^{k_i}\IC({\cO_{2^{2i}1^{2n+1-4i}}},\cE_{2i}^{\oplus m^i_{a_i}})[k_i-2a_i]=0$. Thus $k_i\leq n-2i$. It remains to show that
\beqn
m^i_0=m^i_{n-2i}=1.
\eeqn
We argue using induction on $1\leq i\leq [\frac{n}{2}]$. Consider $i=[\frac{n}{2}]$. If $n$ is even, then it is easy to see that $\IC(\cO_{2^n1},\cE_n)$ is a direct summand. Assume that $n$ is odd. Take $j=n-1$ in \eqref{eqn-cohomology2} we get 
$
H^{2k-2n+j+2}(\upsilon_0^{-1}(x_j),\bC)\neq 0\text{ if and only if }2k=n+1,n-1
$
(we have $\upsilon_0^{-1}(x_{n-1})$ is a nonsingular quadric in $\bP^2$)
and $H^{2k-2n+j+2}(\upsilon_0^{-1}(x_j),\bC)$ gives us a one-dimensional non-trivial local system when $2k=n+1,n-1$.
In view of \eqref{directsummand}, \eqref{eqn-stalk2} and \eqref{eqn-cohomology2} we conclude that
\beqn
\IC(\cO_{2^{n-1}1^3},\cE_{n-1})[-1]\oplus\IC(\cO_{2^{n-1}1^3},\cE_{n-1})[1]\text{ is a direct summand of $\upsilon_*\bC[-]$.}
\eeqn
This proves \eqref{directsummand2} for $i=[\frac{n}{2}]$. By induction hypothesis, suppose that \eqref{directsummand2} holds for all $j<i$. Let $2k=(n-2i)(n-2i+1)+2n-2i-2$. Take the stalk at $x_{2i}$ of $\mH^{2k-n^2-2n+2}$ in \eqref{directsumdec}, we get
\beqn
\dim\mH^{2k-n^2-2n+2}_{x_{2i}}\bigoplus_{a_i=0}^{k_i}\IC({\cO_{2^{2i}1^{2n+1-4i}}},\cE_{2i}^{\oplus m^i_{a_i}})[k_i-2a_i]=1,
\eeqn
\beqn
\text{ i.e. }\dim\mH^{n-2i-\dim\cO_{2^{2i}1^{2n+1-4i}}}_{x_{2i}}\bigoplus_{a_i=0}^{k_i}\IC({\cO_{2^{2i}1^{2n+1-4i}}},\cE_{2i}^{\oplus m^i_{a_i}})[k_i-2a_i]=1,
\eeqn
here we use \eqref{eqn-cohomology2}, \eqref{eqn-stalk2} and the fact that
\beqn
\dim\mH^{2k-n^2-2n+2}_{x_{2i}}\bigoplus_{j=i+1}^{[\frac{n}{2}]}\bigoplus_{a_j=0}^{n-2j}\IC({\cO_{2^{2j}1^{2n+1-4j}}},\cE_{2j}^{\oplus m^j_{a_j}})[n-2j-2a_j]=0,
\eeqn
as $2k-n^2-2n+2+2j(2n+1-2j)+n-2j-2a_j\geq4(j-i)(n+1-i-j)>0$. Thus we can conclude now that
$
m^i_{n-2i}=1\text{ which implies }m_0^i=1.
$ This completes the proof of \eqref{directsummand2}.

Comparing with \eqref{decomposition-dual}, we conclude that
\beqn
\mathfrak{F}(\IC({\cO_{2^{2i}1^{2n+1-4i}}},\cE_{2i}))=
\IC(\Lg_1,\mF_{2i})
\eeqn
and
$
m_{a_i}^i=1\text{ for all }0\leq a_i\leq n-2i.
$
This finishes the proof of Theorem \ref{non-trivial even}.

Finally taking Fourier transform of (\ref{directsumdec}) we obtain  the decomposition in (\ref{decomp-1}).

\begin{corollary}\label{coro odd dimen vanishing}
\begin{enumerate}
\item We have $\mH^k_{x_j}\IC(\cO_{2^i1^{2n-2i+1}},\bC)=0$ for $k$ odd and $j\leq i$.
\item We have $\mH^k_{x_j}\IC(\cO_{2^{2i}1^{2n-4i+1}},\cE_{2i})=0$ for $k$ odd and $j\leq i$.
\end{enumerate}
\end{corollary}
\begin{proof}
 (1) follows from equation \eqref{eqn-stalk}, the fact that $n^2+n$ is even and the fact that $H^{\text{odd}}(\upsilon_0^{-1}(x_j),\bC)=0$. Taking $\mH^k_{x_j}$ on both sides of the equation \eqref{decomp-1}, we see that (2) follows from the fact that $n^2+n$ is even and the fact that $H^{\text{odd}}(\upsilon^{-1}(x_j),\bC)=0$.
\end{proof}

\begin{remark}
For $x\in\grs$, the fibers \[Z_x:=\check\upsilon^{-1}_2(x)=
\{(0\subset V_n\subset V_n^\bot\subset V)\,|\,\on{rank}(\bar x:V_n\ra V/V_n^\bot)\leq 1\}
\]
are the \emph{over-generalized Kummer varieties} introduced in \cite[p.80]{Re}.
For example, when $n=2$, the map $f_x:F_x\ra Z_x$ 
in (\ref{Kummer construction}) realizes $Z_x$ as the quotient $F_x/\sigma$ of 
the Fano variety $F_x$ of lines in  the complete intersection  $\tilde Q\cap\tilde Q_x$ of 2 quadrics 
in $\mathbb P^5$ by the involution $\sigma$ (see \S\ref{non-trivial on 2^n}). There are $16$ fixed points of 
$\sigma$ on $F_x$ corresponding to
$16$ singular points of $Z_x\is F_x/\sigma$. The Hessenberg variety 
$H_x:=\check\upsilon^{-1}(x)$ is the blow up of $Z_x$ at those singular points, which is 
isomorphic to
the Kummer K3 surface coming from the 
$\Jac(C_x)$-torsor $F_x$ together with the involution $\sigma$.
\end{remark}

\subsection{Proof of Theorem \ref{matching}}

Let us start with the matching for the $\IC(\cO_{2^i1^{2n-2i+1}},\bC)$'s with $i$ odd.

\begin{proposition}\label{trivial odd}
If $i$ is odd, then $\mathfrak F(\IC(\mO_{2^i1^{2n-2i+1}},\bC))=\IC(\fg_1,\mL_i)$.
\end{proposition}
\begin{proof}
Assume that $2m\leq n+1$. Let us write
$
\cO_m=\cO_{3^{m-1}2^11^{2n-3m+2}}.
$
 Consider the following resolution map
$\tau_m:\widetilde{\cO}_{m}\to\bar\cO_{m},$
where
\beqn
\widetilde{\cO}_{m}=\{(x,0\subset V_{m-1}\subset V_m\subset V_m^\p\subset V_{m-1}^\p\subset V=\bC^{2n+1})\,|\,x\in\Lg_1,\,xV_m=0,\,xV_{m}^\p\subset V_{m-1}\}.
\eeqn
We show that
\beq\label{IC supp 2 in dec}
\IC({\cO_{2^{2m-1}1^{2n-4m+3}}},\bC)\text{ is a direct summand of $\tau_{m*}\bC[-]$.}
\eeq

Recall that $\cO_\lambda\subset\bar\cO_\mu$ if and only if $\mu\geq\lambda$ if and only if $\lambda^t\geq\mu^t$. Thus we have
\beqn
\bar\cO_{3^i2^j1^{2n+1-3i-2j}}\supset\cO_{3^{i'}2^{j'}1^{2n+1-3i'-2j'}} \text{ if and only if }i\geq i',2i+j\geq 2i'+j'.
\eeqn
In particular, $\cO_{2^{2m-1}1^{2n-4m+3}}\subset\bar\cO_{3^i2^j1^{2n+1-3i+2j}}\subset\bar\cO_m$ if and only if $i\leq m-1$ and $2i+j=2m-1$.

We show in \cite[Lemma 2.4]{CVX1}, independently of this paper, that for  $i\in[0,m-1]$ and $x_i\in\cO_{3^i2^{2m-1-2i}1^{2n-4m+3+i}}$, $\tau_m^{-1}(x_i)\cong \on{OGr}(m-1-i,2m-1-2i)$. It is then easy to check that
\beq\label{semismall-1}
2\dim\tau_m^{-1}(x_i)={\codim_{\bar{\cO}_{m}}\cO_{x_i}}=(m-i-1)(m-i).
\eeq
Thus \eqref{IC supp 2 in dec} follows from the discussion above, \eqref{semismall-1} and the decomposition theorem.

 Consider the map 
$$\check{\tau}_m:\{(x,0\subset V_{m-1}\subset V_m\subset V_m^\p\subset V_{m-1}^\p\subset\bC^{2n+1})\,|\,x\in\Lg_1,xV_{m-1}\subset V_m,\,xV_{m}\subset V_{m}^\p\}\to\Lg_1.$$ We have that 
\beq\label{eqn-ft}
\fF(\tau_{m*}\bC[-])\cong\check{\tau}_{m*}\bC[-].
\eeq
By Proposition \ref{FT of order 2 trivial}, $\mathfrak{F}(\IC({\mO_{2^{2m-1}1^{2n-4m+3}}},\bC))=\IC(\fg_1,\mL_i)$ for some $i\in[1,n]$. We show in \cite[Remark 4.9]{CVX1}, again independently of this paper, that among the $\IC(\fg_1,\mL_i)$'s ($i\geq 1$), only $\IC(\fg_1,\mL_{2j-1})$, $1\leq j\leq m$, appear in the decomposition of $\check{\tau}_{m*}\bC[-]$. Thus by induction on $m$, \eqref{IC supp 2 in dec} and \eqref{eqn-ft} imply that
\beqn
\mathfrak{F}(\IC({\mO_{2^{2m-1}1^{2n-4m+3}}},\bC))=\IC(\fg_1,\mL_{2m-1}).
\eeqn
\end{proof}

So Theorem \ref{matching} is now reduced to the statement that for even $i$, 
\[\mathfrak F(\IC(\mO_{2^i1^{2n-2i+1}},\bC))=\IC(\fg_1,\mL_i).\]
For this we 
observe that $\dim\mL_i\neq\dim\mL_j$ for $i\neq j$, hence
it suffices to show that 
the generic rank of $\mathfrak F(\IC(\mO_{2^i1^{2n-2i+1}},\bC))$ 
is equal to $\dim\mL_i$.

We make use of the theory of characteristic cycles to compute the 
generic rank of \linebreak$\mathfrak F(\IC(\mO_{2^i1^{2n-2i+1}},\bC))$.
Recall the following facts about characteristic cycles (see  \cite{KaS}):
\begin{enumerate}
\item  The Fourier transform preserves the characteristic cycle of conic sheaves. 
\item
Let $\mF$ be an irreducible perverse sheave on $\fg_1$ and let 
$r$ be the generic rank of $\mF$. Then  the characteristic cycle of $\mF$, denoted by $\on{CC}(\mF)$, satisfies  
\[\on{CC}(\mF)=r\cdot (T^*_{\fg_1}\fg_1)+\cdot\cdot\cdot\]
\end{enumerate}

We will prove the following equality of characteristic cycles.
\begin{proposition}\label{CC}
Assume that $i$ is even.
We have 
\[
\on{CC}(\IC(\cO_{2^i1^{2n-2i+1}},\bC))=\on{CC}(\IC(\cO_{2^i1^{2n-2i+1}},\mE_i))+\on{CC}(\IC(\cO_{2^{i-1}1^{2n-2i+3}},\bC)).
\]
\end{proposition}
In Theorem \ref{non-trivial even} we have shown that for {\em even} $i$, the generic rank of $\fF(\IC(\cO_{2^i1^{2n-2i+1}},\mE_i))$ is $\dim\mF_i$. This together with the equality above imply
 Theorem \ref{matching}.
Indeed, using Propositions \ref{trivial odd} and \ref{CC} we see that
\beqn
\on{CC}(\mathfrak F(\IC(\cO_{2^i1^{2n-2i+1}},\bC)))=\on{CC}(\fF(\IC(\cO_{2^i1^{2n-2i+1}},\mE_i)))+\on{CC}(\fF(\IC(\cO_{2^{i-1}1^{2n-2i+3}},\bC)))
\eeqn
\beqn
=\on{CC}(\IC(\fg_1,\mF_i))+\on{CC}(\IC(\fg_1,\mL_{i-1}))
=r\cdot (T^*_{\fg_1}\fg_1)+\cdot\cdot\cdot
\eeqn
where \[r=\dim\mF_i+\dim\mL_{i-1}=
\binom{2n}{ i}-\binom{2n}{i-2}+\binom{2n+1}{ i-1}=\binom{2n+1}{ i}
=\dim\mL_i.\]
Hence the generic rank of $\mathfrak F(\IC(\cO_{2^i1^{2n-2i+1}},\bC))$ is equal to $\dim\mL_i$, 
and the theorem follows.

It remains to prove Proposition \ref{CC}. The proof is given in the next subsection.

\subsection{Proof of Proposition \ref{CC}}\label{sec-proof of proposition cc}

Recall that for any variety $X$ and $\mF\in D(X)$ we can consider the corresponding 
local Euler characteristic function 
\[\chi(\mF):X\ra \bZ\]
defined by 
\[\chi(\mF)_x=\sum(-1)^i\dim(\mH^i_x\mF).\]
We have the following fact
\begin{fact}[\cite{KaS},\ Theorem 9.7.11]
Let $\mF_1,\mF_2\in D(X)$. Then $\on{CC}(\mF_1)=\on{CC}(\mF_2)$
if and only if $\chi(\mF_1)=\chi(\mF_2)$.
\end{fact}
By the fact above we are reduced to show the following:
\begin{prop}\label{CC for SP}
For even $i$, we have 
\[\chi(\IC(\mO_{2^i1^{2n-2i}},\bC))=\chi(\IC(\mO_{2^i1^{2n-2i}},\mE_i))+\chi(\IC(\mO_{2^{i-1}1^{2n-2i+2}},\bC)).\]
\end{prop}

In \S\ref{sec-stalks} Theorem \ref{SO=SP} and Theorem \ref{SO=SP non-trivial local}, we show that
\beq\label{eqn-stalk trivial}
\chi(\IC(\cO_{2^i1^{2n-2i+1}},\bC))=\chi(\IC(\mO'_{2^i1^{2n-2i}},\bC))\text{ for all }i
\eeq  
\beq\label{eqn-stalks nontrivial}
\chi(\IC(\cO_{2^i1^{2n-2i+1}},\cE_{i}))=\chi(\IC(\mO'_{2^i1^{2n-2i}},\cE_i'))\text{ for even }i,
\eeq where $\mO'_{2^i1^{2n-2i}}$ is the nilpotent $Sp(2n)$-orbit in $\mathfrak{sp}(2n)$ corresponding to the partition $2^i1^{2n-2i}$ and $\cE_i'$ is the unique non-trivial irreducible $Sp(2n)$-equivariant local system on $\mO'_{2^i1^{2n-2i}}$. 

As such, we prove the proposition above making use of the classical Springer correspondence 
for symplectic group $Sp(2n)$. As discussed earlier, this will complete the proof of Theorem \ref{matching}.

\begin{proof}[Proof of Proposition \ref{CC for SP}]
The proof is reduced to proving the following: 
for even $i$, $j\leq i$, and $x_j'\in\cO'_{2^j1^{2n-2j}}$ we have
 \beq\label{goal}
 \chi(\IC({\cO_{2^i1^{2n-2i}}'},\bC))_{x_j'}=
 \chi(\IC({\cO_{2^i1^{2n-2i}}'},\cE'_i))_{x_j'}+\chi(\IC({\cO_{2^{i-1}1^{2n-2i+2}}'},\bC))_{x_j'}.
 \eeq
We will show that
\begin{eqnarray}
&&\chi(\IC({\cO'_{2^i1^{2n-2i}}},\bC))_{x_j'}=\chi(\IC({\cO'_{2^{i-j}1^{2n-2i}}},\bC))_{0}
\label{eu0}
\\
&&
\chi(\IC({\cO'_{2^{i-j}1^{2n-2i}}},\bC))_0=
\binom{n-j}{ [\frac{i-j}{2}]}\label{eu1}\\&&\chi(\IC({\cO'_{2^{2i}1^{2n-4i}}},\cE_{2i}))_{x_j'}=\binom{n-j}{ i-\frac{j}{2}}-\binom{n-j}{ i-\frac{j}{2}-1}\text{ if }j\text{ is even},\label{eu2}\\
&&\chi(\IC({\cO'_{2^{2i}1^{2n-4i}}},\cE_{2i}'))_{x_j'}=0\text{ if }j\text{ is odd}.\label{vanishing of Euler characteristic}
\end{eqnarray}
The equality (\ref{goal}) follows.
\end{proof}

\begin{proof}[Proof of \eqref{eu0}] The equality \eqref{eu0} follows from \eqref{eqn-stalk trivial} and Proposition \ref{reduction to zero} (see \S\ref{sec-stalks}), which states that $\mH^k_x\IC(\cO_{2^i1^{2n-2i+1}},\bC)=\mH^{k+s_j}_0\IC(\cO_{2^{i-j}1^{2n-2i+1}},\bC)$, $s_j=j(2n+1-j)$.
Indeed, since $s_j$ is even, we have
\beqn
\chi(\IC({\cO'_{2^i1^{2n-2i}}},\bC))_{x_j'}=\sum(-1)^k
\dim\mH^k_{x_j'}\IC(\mO'_{2^i1^{2n-2i}},\bC)
\eeqn
\beqn=
\sum
(-1)^{k+s_j}\dim\mH^{k+s_j}_0\IC(\mO'_{2^{i-j}1^{2n-2i}},\bC)=\chi(\IC({\cO'_{2^{i-j}1^{2n-2i}}},\bC))_{0}.
\eeqn
\end{proof}

\begin{proof}[Proof of \eqref{eu1} and \eqref{eu2}]
We work in the classical Springer correspondence setting for $\mathfrak{sp}(2n)$.  
Let $\cN=\cN_{\mathfrak{sp}(2n)}$ and $G=Sp(2n)$ in this proof. Recall that we have the Springer resolution $\varphi:\widetilde{\cN}\to\cN$ and
 \beq\label{sp resolution}
 \varphi_*\bC[-]\cong\bigoplus_{(\mO,\cE)}\IC({\mO},\cE)\otimes V_{\mO,\cE}.
 \eeq
Here $(\mO,\cE)$ runs through the pairs consisting of a nilpotent $G$-orbit $\mO\subset \mN$ and an irreducible $G$-equivarant local system on $\mO$, that appear in the Springer correspondence. Moreover $V_{\mO,\cE}$ denotes the irreducible Weyl group representation corresponding to the pair $(\mO,\cE)$ under Springer correspondence. It follows that
\beq\label{multip}
\chi(\IC({\mO},\cE))_x=[V_{\mO,\cE}:H^*(\mB_x,\bC)]
\eeq
for $x\in\bar{\mO}$. Here $\mB_x=\varphi^{-1}(x)$ is the Springer fiber and $[V_{\mO,\cE}:H^*(\mB_x,\bC)]$  denotes the multiplicity of $V_{\mO,\cE}$ in the Weyl group representation $H^*(\mB_x,\bC)$.

Let us denote the Weyl group of type $C_n$ by $W_n$. It is well known that the irreducible representations of $W_n$ are parametrized by pairs of partitions $(\alpha)(\beta)$ such that $|\alpha|+|\beta|=n$. Lusztig in \cite{Lu} has computed the (generalized) Springer correspondence explicitly. In our case, we have
\begin{eqnarray}
&& V_{(\cO_{2^{2i}1^{2n-4i}}',\bC)}=(1^i)(1^{n-i}),\ V_{(\cO_{2^{2i-1}1^{2n-4i+2}}',\bC)}=(1^{n-i+1})(1^{i-1}),\label{spc-1}\\  
&&V_{(\cO_{2^{2i}1^{2n-4i}}',\cE_{2i}')}= (0)(2^i1^{n-2i})\label{spc-2},
\end{eqnarray}
and $\IC({\cO'_{2^{2i-1}1^{2n-4i+2}}},\cE_{2i-1}')$ does not appear in the decomposition of $\varphi_*\bC[-]$.
It follows from \eqref{multip} and \eqref{spc-1} that
\beqn
\chi(\IC({\cO_{2^i1^{2n-2i}}'},\bC))_0=[V_{(\cO_{2^i1^{2n-2i}}',\bC)}:H^*(\cB,\bC)]=\dim V_{({\cO_{2^i1^{2n-2i}}'},\bC)}=\binom{n}{[\frac{i}{2}]}.
\eeqn
This proves \eqref{eu1}.

We prove \eqref{eu2}.
Recall that  $H^{\text{odd}}(\cB_x,\bC)=0$ (see \cite{CLP}) and if $x\in \Lie\,L$ is regular nilpotent, we have (see \cite{Lu2})
\beq\label{eqn-induced rep}
\sum H^{2k}(\cB_x,\bC)=\Ind_{W_L}^{W_n}\bC,
\eeq
where $L$ is a Levi subgroup of $G$ and $W_L$ is the Weyl group of $L$.

Consider our elements $x_j'\in\cO'_{2^j1^{2n-2j}}$, where $j\leq 2i$. Assume that $j=2j_0$. We can find a Levi subgroup $L\subset Sp(2n)$ such that  $L\cong {GL(2)\times\cdots\times GL(2)}\ ({j_0}\text{ copies})\subset L_0=GL(n)$, and $x_j'\in\Lie\, L$ is regular, where $L_0$ is a Levi subgroup of a maximal parabolic subgroup. We have
$
W_L\cong {S_2\times \cdots\times S_2}\ ({j_0}\text{ copies})\subset S_n.
$
Thus
\begin{eqnarray*}
&\chi(\IC(\cO'_{2^{2i}1^{2n-4i}},\cE'_{2i}))_{x_j'}\stackrel{\eqref{multip}}=[V_{(\cO'_{2^{2i}1^{2n-4i}},\cE'_{2i})},H^*(\cB_{x_j'},\bC)]\stackrel{\eqref{spc-2}\  \eqref{eqn-induced rep}}=[(0)(2^i1^{n-2i}),\Ind _{W_L}^{W_n}\bC]\\
&=[(0)(2^i1^{n-2i}),\Ind _{S_n}^{W_n}\Ind_{W_L}^{S_n}\bC]=[\operatorname{Res}^{W_n}_{S_n}(0)(2^i1^{n-2i}),\Ind_{W_L}^{S_n}\bC]_{S_n}\\
&=[(2^i1^{n-2i}),\Ind_{W_L}^{S_n}\bC]_{S_n}=[2^i1^{n-2i},H^*(\cB_{2^{j_0}1^{n-2j_0}}^{GL(n)},\bC)]_{S_n}=\chi(\IC(\cO_{2^i1^{n-2i}},\bC))_{2^{j_0}1^{n-2j_0}}^{GL(n)}.
\end{eqnarray*}
Here in the last two equalities, $\cB_{2^{j_0}1^{n-2j_0}}^{GL(n)}$ denotes the Springer fiber of an element in the nilpotent orbit corresponding to the partition $2^{j_0}1^{n-2j_0}$ in $\mathfrak{gl}(n)$, and $\chi(\IC(\cO_{2^i1^{n-2i}},\bC))_{2^{j_0}1^{n-2j_0}}^{GL(n)}$ is defined analogously in $GL(n)$.
Now it follows from the classical result for $GL(n)$ that 
\beqn\label{multiplicity3}
\chi(\IC(\cO_{2^i1^{n-2i}},\bC))_{2^{j_0}1^{n-2j_0}}^{GL(n)}=K_{2^i1^{n-2i},2^{j_0}1^{n-2j_0}}=K_{2^{i-j_0}1^{n-2i},1^{n-2j_0}}=\binom{n-2j_0}{ i-j_0}-\binom{n-2j_0}{ i-j_0-1},
\eeqn
where $K_{\lambda,\mu}$ denotes the Kostka number. This proves \eqref{eu2}.
\end{proof}

\begin{proof}[Proof of \eqref{vanishing of Euler characteristic}] Let $x_j\in\cO_{2^j1^{2n-2j+1}}$, where $j$ is {\em odd}. In view of \eqref{eqn-stalks nontrivial}, it suffices to show that
$\chi(\IC({\cO_{2^{2i}1^{2n-4i+1}}},\cE_{2i}))_{x_j}=0.$ We prove this using the decomposition in \eqref{decomp-1} (see \S\ref{matching for 
non trivial 2^n})
\begin{eqnarray*}
&&\upsilon_*\bC[n^2+2n-2]\cong \bigoplus_{a=0}^{n-2}\bigoplus_{i=0}^n\IC({\cO_{2^{i}1^{2n-2i+1}}},\bC)[n-2-2a]\\
&&\oplus\bigoplus_{i=1}^{[\frac{n}{2}]}\bigoplus_{a_i=0}^{n-2i}\IC({\cO_{2^{2i}1^{2n+1-4i}}},\cE_{2i})[n-2i-2a_i]\oplus\bigoplus_{a=0}^n\IC(\cO_{1^{2n+1}},\bC)[n-2a].\nonumber
\end{eqnarray*}
Note that $\mH^{\text{odd}}_{x_j}\IC(\cO_{2^i1^{2n+1-2i}},\bC)=\mH^{\text{odd}}_{x_j}\IC(\cO_{2^{2i}1^{2n+1-4i}},\cE_{2i})=0$ (see Corollary \ref{coro odd dimen vanishing}). Thus $\chi(\IC(\cO_{2^{2i}1^{2n+1-4i}},\cE_{2i}))_{x_j}\geq 0$. It follows from the decomposition above that 
\beqn
\chi(H^*(\upsilon^{-1}(x_j),\bC)=(n-1)\sum_{i=j}^n\chi(\IC(\cO_{2^i1^{2n+1-2i}},\bC))_{x_j}+\sum_{j\leq2i\leq n}\chi(\IC(\cO_{2^{2i}1^{2n+1-4i}},\cE_{2i}))_{x_j}.
\eeqn
We have $\chi(H^*(\upsilon^{-1}(x_j),\bC)=(n-1)2^{n-j}$ (here $\upsilon^{-1}(x_j)$ is a quadric bundle over $\on{OGr}(n-j,2n-2j+1)$ with fibers having euler characteristic $n-1$, see \S\ref{matching for 
non trivial 2^n}).
Thus equations  \eqref{eqn-stalk trivial}  and \eqref{eu1} imply that
\beqn
\sum_{i=j}^n\chi(\IC(\cO_{2^i1^{2n+1-2i}},\bC))_{x_j}=\sum_{i=j}^n\binom{n-j}{ [\frac{i-j}{2}]}=2^{n-j}.
\eeqn
It follows that
$\chi(\IC(\cO_{2^{2i}1^{2n+1-4i}},\cE_{2i}))_{x_j}=0$.
\end{proof}

The proof of Proposition \ref{CC} is complete. This completes the proof of Theorem \ref{matching}.

\section{Stalks of the IC sheaves}\label{sec-stalks}
In this section we establish an interesting isomorphism between 
stalks of the IC sheaves  supported on nilpotent orbits of order two  and 
stalks of the IC sheaves supported on 
 nilpotent orbits of order two in $\mathfrak{sp}(2n)$.   More precisely, let $\cN_{\mathfrak{sp}(2n)}$ denote the nilpotent cone of $\mathfrak{sp}(2n)$. Let $\cO_{2^i1^{2n-2i}}'\subset \cN_{\mathfrak{sp}(2n)}$ denote the nilpotent orbit corresponding to the partition $2^i1^{2n-2i}$.
We have $\dim
\mO'_{2^i1^{2n-2i}}=\dim\mO_{2^i1^{2n-2i+1}}=i(2n+1-i)$,
and for each $i\geq 1$, there exists a unique nontrivial irreducible $Sp(2n)$-equivariant 
local system $\mE_i'$ on $\mO'_{2^i1^{2n-2i}}$.

We have the 
following theorems, which were used in the proof of Theorem \ref{matching} in \S\ref{sec orbits of order 2}. We remark that the proofs of these theorems are independent of \S\ref{sec orbits of order 2}.
\begin{thm}\label{SO=SP}
For 
$x\in\mO_{2^j1^{2n-2j+1}}$ and $x'\in\mO'_{2^j1^{2n-2j}}$,
we have \[\mathcal H_x^k\IC(\cO_{2^i1^{2n-2i+1}},\bC)\is
\mathcal H_{x'}^k\IC(\cO'_{2^i1^{2n-2i}},\bC).\] 
\end{thm}
\begin{thm}\label{SO=SP non-trivial local}
Assume that $i$ is even.
For
$x\in\mO_{2^j1^{2n-2j+1}}$ and $x'\in\mO'_{2^j1^{2n-2j}}$,
we have 
\[\mathcal H_x^k\IC(\cO_{2^i1^{2n-2i+1}},\mE_i)\is
\mathcal H_{x'}^k\IC(\cO'_{2^i1^{2n-2i}},\mE'_i).\]
\end{thm}
The following proposition was used in the proof of \eqref{eu0}.
\begin{prop}\label{reduction to zero}
Let $j\leq i$ and set \[s_j=\dim\mO_{2^i1^{2n-2i+1}}-\dim\mO_{2^{i-j}1^{2n-2i+1}}=j(2n+1-j).\] Then for 
$x\in\mO_{2^j1^{2n-2j+1}}$,
we have  
\[\mH^k_x\IC(\cO_{2^i1^{2n-2i+1}},\bC)=\mH^{k+s_j}_0\IC(\cO_{2^{i-j}1^{2n-2i+1}},\bC).\]
\end{prop}

\begin{remark}
Note that, in general, the singularities for 
$\bar\mO_{2^i1^{2n-2i+1}}$ and $\bar\mO'_{2^i1^{2n-2i}}$ are non-isomorphic.
For example, the Euler obstruction $\on{Eu}(0,\mO_{min}')$ for the minimal orbit $\mO_{min}':=\mO'_{2^11^{2n-2}}$ is zero
(see  \cite{EM}). On the other hand, using 
\[\on{CC}(\IC(\cO_{min},\bC))=
\on{CC}(\mathfrak F((\IC(\cO_{min},\bC)))=
\on{CC}(\IC(\fg_1,\mathcal L_1))\] and $\on{dim}\mathcal L_1=2n+1$ (see Theorem \ref{matching}),
one
can check that $\on{Eu}(0,\cO_{min})=2n$, where $\cO_{min}:=\mO_{2^11^{2n-1}}$.

\end{remark}

\subsection{Resolutions}\label{resolution}
For the proof of Theorem \ref{SO=SP} and 
Theorem \ref{SO=SP non-trivial local} we need several preliminary steps. 
We begin with the construction of resolutions of $\bar{\cO}_{2^i1^{2n+1-2i}}$ and 
$\bar{\cO}_{2^i1^{2n-2i}}'$.

For the orbit $\cO_{2^i1^{2n+1-2i}}\subset \cN_1$, $1\leq i\leq n$, 
consider Reeder's resolution map of $\bar{\cO}_{2^i1^{2n+1-2i}}$,  
\beqn
\sigma_i:\{(x,0\subset V_i\subset V_i^\p\subset V=\bC^{2n+1})\,|\,x\in\Lg_1,\ xV\subset V_i\}\to\bar{\cO}_{2^i1^{2n+1-2i}}.
\eeqn
For the orbit $\mO_{2^i1^{2n-2i}}'\subset\cN_{\mathfrak{sp}(2n)}$,
we have
the following resolution map for $\bar{\cO}_{2^i1^{2n-2i}}'$ (\cite{Hes})
\beqn
{\tau}_i:\{(x,0\subset U_i\subset U_i^\p\subset U=\bC^{2n})\,|\,x\in\mathfrak{sp}(2n),\ xU\subset U_i\}\to\bar{\cO}'_{2^i1^{2n-i}}.
\eeqn
Here $U$ is a $2n$-dimensional vector space equipped with a non-degenerate symplectic form $\langle,\rangle$ such that $\mathfrak{sp}(2n)=\mathfrak{sp}(U,\langle,\rangle)$, $\dim U_i=i$ and $U_i^\p=\{u\in U\,|\,\langle u,U_i\rangle=0\}$.

We show that for
$x\in\mO_{2^j1^{2n-2j+1}}$ and $x'\in\mO'_{2^j1^{2n-2j}}$, $j\leq i$,
\begin{eqnarray}
&&\text{We have  $\dim\sigma_i^{-1}(x)=\dim\tau_i^{-1}(x')$.
Moreover, $H^*(\sigma_i^{-1}(x),\bC)\is H^*(\tau_i^{-1}(x'),\bC)$.}\label{property1}
\\
&&\text{The action of $A(x):=Z_K(x)/Z_K(x)^0$ (resp. $A'(x'):=Z_{Sp(2n)}(x')/Z_{Sp(2n)}(x')^0$) }\label{property2}\\
&&\quad\text{ on $H^*(\sigma_i^{-1}(x),\bC)$ (resp. $H^*(\tau_i^{-1}(x'),\bC)$)
is trivial.}\nonumber
\end{eqnarray}
Let $s_{ij}=2\dim(\sigma_i^{-1}(x_j))-\on{codim}_{\bar\mO_{2^i1^{2n-2i+1}}}\mO_{2^j1^{2n-2j+1}}=2(i-j)(n-i)$, where $x_j\in\mO_{2^j1^{2n-2j+1}}$. It follows from \eqref{property1} and \eqref{property2} that
\begin{lemma}
We have
\begin{eqnarray}
&&(\sigma_{i})_*\bC[-]=
\bigoplus_{j=0}^i\bigoplus_{k=0}^{s_{ij}} 
\IC(\cO_{2^j1^{2n-2j+1}},\bC^{t_{jk}^i})[\pm k]\label{decomposition sigmai*C}\\
&& (\tau_{i})_*\bC[-]=\bigoplus_{j=0}^i\bigoplus_{k=0}^{s_{ij}}  
\IC(\cO'_{2^j1^{2n-2j}},\bC^{(t_{jk}^i)'})[\pm k],\label{decomposition tau}
\end{eqnarray}
here $t_{ik}^i=(t_{ik}^i)'=\delta_{k,0}$.
\end{lemma}

We prove \eqref{property1} and \eqref{property2}. For \eqref{property1}, we show that
\beq\label{eqn-fibers}
\sigma_i^{-1}(x)\cong\on{OGr}(i-j,2n-2j+1)\quad\text{and }\tau_i^{-1}({x}')\cong\on{SpGr}(i-j,2n-2j).
\eeq
For \eqref{property2}, we show that the action of $A(x)$ (resp. $A'(x')$)
 on $\sigma_i^{-1}(x)$ (resp. $\tau_i^{-1}(x')$) is trival, thus inducing a trivial action on $H^*(\sigma_i^{-1}(x),\bC)$
  (resp.  $H^*(\tau_i^{-1}(x'),\bC)$). 

Let us consider the case of $x\in\cO_{2^j1^{2n+1-2j}}$ first. The following lemma allows us to choose convenient basis for $V$.
\begin{lemma}\label{lemma-basis}
Let $x\in\cO_\lambda\subset\cN_1$, where $\lambda=(\lambda_1,\lambda_2,\ldots,\lambda_s)\in\mathbf{P}(2n+1)$. There exist $v_i\in V$, $i\in[1,s]$ such that $V=\Span\{x^{a_i}v_i,\ a_i\in[0,\lambda_i-1],\ i\in[1,s]\}$ and
\begin{eqnarray*}
&&x^{\lambda_i}v_i=0,\,\langle x^{a_i}v_i,x^{b_j}v_j\rangle_Q=\delta_{i,j}\delta_{a_i+b_j,\lambda_i-1},\ i,j\in[1,s].
\end{eqnarray*}
\end{lemma}
\begin{proof}We prove that the lemma holds also for even dimensional $V$ (the fact that some orbits when $\dim V$ is even are parametrized by the same partition does not affect the proof).
We prove by induction on $\dim\,V$. It is clear when $\dim\, V=1$. Assume that $\dim\,V>1$. One checks easily that there exists $v\in V$ such that $\langle v,x^{\lambda_1-1}v\rangle_Q\neq0$. We can assume that $\langle v,x^{\lambda_1-1}v\rangle_Q=1$. We have $x^{\lambda_1}v=0$ and $v,xv,\ldots,x^{\lambda_1-1}v$ are linearly independent. Let $W=\on{Span}\{v,xv,\ldots,x^{\lambda_1-1}v\}$. Then $V=W\oplus W^\p$ as $\langle,\rangle_Q|_{W}$ is non-degenerate. Note that $W^\p$ is $x$-stable and $x|_{W^\p}\in \cO_{\lambda'}$, where $\lambda'=(\lambda_2,\ldots,\lambda_s)$. The lemma follows by induction. 
\end{proof}
Take a basis $\{v_l,\ xv_l,\  l\in[1,j],\ w_s,\ s\in[1,2n+1-2j]\}$ of $V$ as in Lemma \ref{lemma-basis}. In terms of this basis, the fiber $\sigma_i^{-1}(x)$ can be described as follows. It consists of the flags $0\subset V_i\subset V_i^\perp\subset \bC^{2n+1}$, where $V_i=\Span\{xv_k,\ k\in[1,j]\}\oplus W_{i-j}$ and $W_{i-j}\subset\Span\{w_s,\  s\in[1,2n+1-2j]\}$ is such that $W_{i-j}\subset W_{i-j}^\p$. Thus \eqref{eqn-fibers} holds in this case. Define $g\in Z_K(x)$ as follows. When $j$ is odd, let
$
gv_k=-v_k,\ gw_s=w_s;
$
when $j$ is even, let
$gv_1=v_2,\ gv_2=v_1,\ gv_k=v_k,\ k\neq 1,2,\ gw_s=w_s.$ (Note that the actions of $g$ on other basis vectors are determined by the property that $g\in Z_K(x)$.) It is easy to see that $g\notin Z_K(x)^0$. Thus we can identify $A(x)$ with $\{g,1\}$. Now it is easy to see that $A(x)$ fixes each flag in $\sigma_i^{-1}(x)$, thus acting trivially.

The proof for the case of $x'\in\cO'_{2^j1^{2n-2j}}$ is entirely similar. There exist vectors\linebreak 
$
v_1',\ldots,v_j',w_1',\ldots,w_{n-j}',u_1',\ldots,u_{n-j}'\in U=\bC^{2n}$ such that $x^{'2}v_l'=0,\  l\in[1, j],\ x'w_s'=x'u_s'=0,\  s\in[1,n-j]$, $U=\on{Span}\{v_l',x'v_l',\ l\in[1,j], w_s',u_s',s\in[1,n-j]\}$, and
\begin{eqnarray*}
&&\ \langle v_k',v_l'\rangle=\langle v_k',w_s'\rangle=\langle v_k',u_s'\rangle=\langle w_s',w_t'\rangle=\langle u_s',u_t'\rangle=0,\ \langle v_k',x'v_l'\rangle=\delta_{k,l},\ \langle w_s',u_t'\rangle=\delta_{s,t}.
\end{eqnarray*}
(Note that $\langle x'v,w\rangle=-\langle v,x'w\rangle$ as $x'\in\mathfrak{sp}(2n)$.) In terms of this basis, the fiber $\tau_i^{-1}(x')$ can be described as follows. It consists of the flags $0\subset U_i\subset U_i^\perp\subset \bC^{2n}$, where $U_i=\Span\{x'v_k',1\leq k\leq j\}\oplus W_{i-j}'$ and $W_{i-j}'\subset\Span\{w_s',u_s',1\leq s\leq n-j\}$ is such that $W_{i-j}'\subset W_{i-j}^{'\p}$. Define $g'\in Z_{Sp(2n)}(x')$ as follows. When $j$ is odd, let
$
g'v_k'=-v_k',\ gw_s'=w_s',\ gu_s'=u_s';
$
when $j$ is even, let
$g'v_1'=v_2',\ g'v_2'=v_1',\  g'v_k'=v_k', k'\neq 1,2,\ g'w_s'=w_s',g'u_s'=u_s'.$ As before we can identify $A'(x')$ with $\{g',1\}$and it is easy to see that $A'(x')$ fixes each flag in $\tau_i^{-1}(x')$, thus acting trivially.

\subsection{The maps $\widetilde{\sigma}_i$ and $\widetilde{\tau}_i$ }\label{tilde tau, sigma}We preserve the notations from \S\ref{resolution}. 
For the proof of Theorem \ref{SO=SP non-trivial local}, we need 
the following auxiliary maps. 

For $\mO_{2^i1^{2n-2i+1}}\subset\cN_1$,
consider the map 
\beqn
\widetilde{\sigma}_i:\{(x,0\subset V_{i-1}\subset V_i\subset V_i^\p\subset V_{i-1}^\p\subset V)\,|\,x\in\Lg_1,\ xV\subset V_i,\ xV_{i-1}^\p\subset V_{i-1}\}\to\bar{\cO}_{2^i1^{2n+1-2i}}.
\eeqn
For $\mO'_{2^i1^{2n-2i}}\subset\cN_{\mathfrak{sp}(2n)}$,
consider the following map
\begin{eqnarray*}
\widetilde{\tau}_i:\{(x',0\subset U_{i-1}\subset U_i\subset U_i^\p\subset U_{i-1}^\p\subset U)\,|\,x'\in\mathfrak{sp}_{2n},\,x'U\subset U_i,\,x'U_{i-1}^\p\subset U_{i-1}\}\to\bar{\cO}'_{2^i1^{2n-2i}}.
\end{eqnarray*}

We show that for
$x\in\mO_{2^j1^{2n-2j+1}}$ and $x'\in\mO'_{2^j1^{2n-2j}}$, $j\leq i$,
\beq \label{property3}
\begin{gathered}
\text {$\dim\tilde\sigma_i^{-1}(x)=\dim\tilde\tau_i^{-1}(x')$, $H^*(\tilde\sigma_i^{-1}(x),\bC)\is H^*(\tilde\tau_i^{-1}(x'),\bC)$, and the latter }\\
\text{isomorphism is compatible with the actions of $A(x)$ and $A'(x')$.}\end{gathered}
\eeq
\beq\label{property4}
\begin{gathered}
\text{$A(x)$ (resp. $A'(x')$) acts trivially on }\text{$H^*(\tilde\sigma_i^{-1}(x),\bC)$ (resp. $H^*(\tilde\tau_i^{-1}(x'),\bC)$), $i$ odd.}
\end{gathered}
\eeq
It follows that
\begin{lemma}We have
\begin{eqnarray}
&&(\tilde\sigma_i)_*\bC[-]=\bigoplus_{j,k\geq 0} 
\IC(\cO_{2^j1^{2n-2j+1}},\bC^{m^i_{jk}})[\pm k]\bigoplus_{j\text{ even},\ k\geq 0}
\IC(\cO_{2^j1^{2n-2j+1}},\mE_j^{a^i_{jk}})[\pm k]\label{decomposition sigmatilde}\\
&&(\tilde\tau_i)_*\bC[-]=\bigoplus_{j,k\geq 0}
\IC(\cO'_{2^j1^{2n-2j}},\bC^{(m^i_{jk})'})[\pm k]\bigoplus_{j\text{ even},\ k\geq 0}
\IC(\cO'_{2^j1^{2n-2j}},(\mE_j')^{(a^i_{jk})'})[\pm k].\label{decomposition tautilde}
\end{eqnarray}
\end{lemma}

We prove \eqref{property3} and \eqref{property4} in the reminder of this subsection.

The fiber $\widetilde{\sigma}_i^{-1}(x)$ is a quadric bundle over
$
\sigma_i^{-1}(x)\cong\on{OGr}(i-j,2n-2j+1)
$
with fibers a quadric $\sum_{s=1}^jb_s^2=0$ in $\bP^{i-1}=\{[b_1:b_2:\cdots:b_i]\}$.
More precisely, we have an obvious map $\pi:\widetilde{\sigma}_i^{-1}(x)\to\sigma_i^{-1}(x)$,
\beqn
(0\subset V_{i-1}\subset V_i\subset V_i^\p\subset V_{i-1}^\p\subset V):=(V_{i-1}\subset V_i)\mapsto (0\subset V_i\subset V_i^\p\subset V):=(V_i)
\eeqn
by forgetting $V_{i-1}$. Now we describe the fibers of $\pi$. Recall that if $(V_i)\in\sigma_i^{-1}({x})$, then there exists $W_{i-j}\subset\Span\{w_k,\,k\in[1, 2n+1-2j]\}$ with $W_{i-j}\subset W_{i-j}^\p$ such that
$
V_i=\Span\{xv_k,\,k\in[1,j]\}\oplus W_{i-j}.
$ 
Let $[b_1:b_2:\cdots:b_i]$ be the homogenous coordinates of $\bP(V_i)$ given by the basis $\{xv_k,\,k\in[1,j],\,\tilde{w_l},\,l\in[1,i-j]\}$ of $V_i$, where $\{\tilde{w_l}\}$ is a basis of $W_{i-j}$. It is easy to check that the fibers of 
$\pi$ are isomorphic to the quadric $Q:\sum_{s=1}^jb_s^2=0$ in $\bP(V_i)\cong\bP^{i-1}$. It follows that
\beqn
H^*(\widetilde{\sigma}_i^{-1}(x),\bC)\cong H^*(Q,\bC)\otimes H^*(\sigma_i^{-1}(x),\bC).
\eeqn
We describe the action of $A(x)$ on $H^*(\widetilde{\sigma}_i^{-1}(x),\bC)$. As we have shown that $A(x)$ acts trivially on $H^*(\sigma_i^{-1}(x),\bC)$, it suffices to describe the action of $A(x)$ on $H^*(Q,\bC)$.

We claim that if $j$ is odd, then $A(x)$ acts on $H^*(Q,\bC)$ trivially, thus acting trivially on $H^*(\widetilde{\sigma}_i^{-1}(x),\bC)$, and
if $j$ is even, then $A(x)$ acts on $H^{2k}(Q,\bC)$ trivially if $2k\neq 2i-j-2$ and $H^{2i-j-2}(Q,\bC)\cong\bC\oplus\cE$. This follows from the following lemma.

\begin{lemma}
Let $Q$ be the quadric given by the equation $\sum_{i=1}^k a_i^2 = 0$ in $\bP^{n-1}$ with coordinates $[a_1, \dots, a_n]$ and consider the automorphism $\gamma$ of $\bP^{n-1}$ given by $\gamma[a_1, a_2, a_3, \dots, a_n] = [a_2, a_1, a_3, \dots, a_n]$. If $k$ is odd, then $\gamma$ acts trivially on $H^*(Q,\bC)$. If $k$ is even, $\gamma$ acts trivially on $H^{2j}(Q)$ for $j\neq 2n-k-2$ and the action on the two dimensional space $H^{2n-k-2}(Q,\bC)$ has eigenvalues $1$ and $-1$.
\end{lemma}
\begin{proof}
The quadric $Q$ is the join of the nonsingular quadric $\tilde Q$ in $\bP^{k-1}\subset \bP^{n-1}$ given by $\sum_{i=1}^k a_i^2 = 0$ and the linear subspace $L$ of dimension $n-k-1$ given by $a_{1} =  \dots = a_k = 0$. Now $Q-L$ is an affine space bundle over $\tilde Q$ of fiber dimension $n-k$. Thus, $H^i_c(Q-L,\bC)= H^{i-2n+2k}(\tilde Q,\bC)$. As $\tilde Q$ and $L$ only have (compactly supported) cohomology in even degrees we conclude that
\begin{equation*}
H^i(Q,\bC)= H^{i-2n+2k}(\tilde Q,\bC)\oplus H^i(L,\bC)
\end{equation*}
The automorphism $\gamma$ of course acts trivially on the cohomology of $\bP^{n-1}$ and hence it acts trivially on the cohomology of $L$ (it even acts trivially on $L$ itself). Thus we are reduced to consider the action of $\gamma$ on the cohomology of $\tilde Q$. The action on the non-primitive cohomology is trivial and so the only possibly nontrivial action is on $H^{k-2}_{prim}(\tilde Q,\bC)$. If $k$ is odd, this group is zero and so we are reduced to the case of $k$ even. 

Assume that $k$ is even and $k=2k_0$. The variety of $(k_0-1)$-planes contained in $\tilde{Q}$ has two disjoint irreducible components, which can be identified with the two disjoint irreducible components of the variety of maximal isotropic spaces in $\bC^{2k_0}$ (equipped with the standard bilinear form). It is clear that $\gamma$, regarded as an element in $O_{2k_0}-SO_{2k_0}$, interchanges these two irreducible components. Now Reid in \cite[Theorem 1.12]{Re} has shown that $H^{k-2}(\tilde{Q},\bC)=\operatorname{span}(a,b)$, where $a$ and $b$ are the classes of $(k_0-1)$-planes from the two families respectively. Thus our lemma follows. 
\end{proof}

The fiber $\widetilde{\tau}_i^{-1}({x}')$ is a quadric bundle over
$
\tau_i^{-1}(x')\cong\on{SpGr}(i-j,2n-2j)
$
with fibers a quadric $\sum_{s=1}^jb_s^2=0$ in $\bP^{i-1}$. The action of $A'(x')$ on $H^*(\widetilde{\tau}_i^{-1}({x}'),\bC)$ is entirely similar. One checks readily that~\eqref{property3} and \eqref{property4} hold.

\subsection{Proof of Theorem \ref{SO=SP}, Theorem \ref{SO=SP non-trivial local} and Proposition \ref{reduction to zero}} The proofs  are based on some well-known principle by making use of the resolutions $\sigma_i$, $\tau_i$ of 
$\bar\mO_{2^i1^{2n-2i+1}}$, $\bar\mO'_{2^i1^{2n-2i}}$  defined 
in \S\ref{resolution}, and the maps $\tilde\sigma_i$, $\tilde\tau_i$  introduced in \S\ref{tilde tau, sigma} for $i$ even. We give the detailed proof for Theorem \ref{SO=SP} in the following. The proofs of Theorem \ref{SO=SP non-trivial local} and Proposition \ref{reduction to zero}
are entirely similar and we leave that to the readers. We mention only that the proof of Proposition \ref{reduction to zero} indicates the following 
\beq\label{lem-decomposition numbers reduction}
\text{$(t_{l,k}^i)_n=(t_{l-j,k}^{i-j})_{n-j}$ for $j\leq l< i$.}
\eeq
Here the numbers are written as $(t_{l,k}^i)_n$ to indicate that the ambient symmetric pair is $(SL(2n+1), SO(2n+1))$ and they are defined in~\eqref{decomposition sigmai*C}.

To prove Theorem \ref{SO=SP}, we begin with the following lemma.
\begin{lemma}\label{stalks are equal}
 Let $i\in[1,n]$. Assume that $t_{jk}^{i'}=(t_{jk}^{i'})'$ for all $j,k$, and all $i'\leq i$ in \eqref{decomposition sigmai*C} and \eqref{decomposition tau}. Then we have 
\[\mathcal H_x^l\IC(\cO_{2^i1^{2n-2i+1}},\bC)\is
\mathcal H_{x'}^l\IC(\cO'_{2^i1^{2n-2i}},\bC)\] for 
$x\in\mO_{2^j1^{2n-2j+1}}$, $x'\in\mO'_{2^j1^{2n-2j}}$ and $j\leq i$.

\end{lemma}
\begin{proof}
We prove the lemma  by induction on $i$. The case when $i=1$ is clear. In fact, we only need to check the conclusion of the lemma for $j=0$. We have
\beqn
(\sigma_1)_*\bC[-]=\IC(\mO_{2^11^{2n-1}},\bC)\oplus\bigoplus_{k\geq 0}
\IC(\mO_{1^{2n+1}},\bC^{t_{0k}^1})[\pm k].
\eeqn
\beqn
(\tau_1)_*\bC[-]=\IC(\mO'_{2^11^{2n-2}},\bC)\oplus\bigoplus_{k\geq 0}
\IC(\mO_{1^{2n}},\bC^{(t_{0k}^1})')[\pm k].
\eeqn
It is clear that $\mH^l_0\IC(\mO_{2^11^{2n-1}},\bC)\cong \mH^l_0\IC(\mO'_{2^11^{2n-2}},\bC)$ as they are determined by the cohomology of $\sigma_1^{-1}(0)\cong\tau_1^{-1}(0)$ and the numbers $(t_{0k}^1)=(t_{0k}^1)'$ in the same way. 

By induction hypothesis, we can assume  that for $s<i$
\beq\label{induction s}
\mathcal H_x^k\IC(\cO_{2^{s}1^{2n-2s+1}},\bC)\is
\mathcal H_{x'}^k\IC(\cO'_{2^s1^{2n-2s}},\bC).
\eeq
Recall that
\[(\sigma_i)_*\bC[-]=\IC(\mO_{2^i1^{2n-2i+1}},\bC)\oplus\bigoplus_{s<i}
\IC(\mO_{2^s1^{2n-2s+1}},\bC^{t_{sk}^i})[\pm k].\]
This implies that the stalks of $\IC(\mO_{2^i1^{2n-2i+1}},\bC)$
are uniquely determined by 
the stalks of 
$\IC(\mO_{2^s1^{2n-2s+1}},\bC)$ for $s<i$, 
the cohomology groups 
of the fibers of the map $\sigma_i$, and 
the numbers $t_{sk}^i$, $s<i$. 
Similarly, the stalks of $\IC(\mO'_{2^i1^{2n-2i}},\bC)$ are 
uniquely determined, in the same way, by the stalks of 
$\IC(\mO'_{2^s1^{2n-2s}},\bC)$ for $s<i$,  the cohomology groups 
of the fibers of the map $\tau_i$, and the numbers $(t_{sk}^i)'$, $s<i$. 
Now the desired claim follows form 
(\ref{induction s}), \eqref{property1}, and the assumption that $t^i_{jk}=(t_{jk}^i)'$.
\end{proof}

By the lemma above, to prove Theorem \ref{SO=SP}, it suffices to show the following 
\begin{lemma}\label{decomposition n}
We have $t^i_{jk}=(t^i_{jk})'$. 
\end{lemma}
\begin{proof}We argue by induction on $i$. The case when $i=1$ is easy to check, i.e. we have that $t^1_{0,k}=(t^1_{0,k})'=1$ when $k\in[0,s_{1j}]$ and $k$ is even, and $t^1_{0k}=(t^1_{0k})'=0$ otherwise. 

So by induction hypothesis we can assume that for all $s< i$, 
$t^s_{jk}=(t^s_{jk})'$. By lemma \ref{stalks are equal}, we have for $s<i$,
\beq\label{equality for stalks 1)}
\mathcal H_x^k\IC(\cO_{2^{s}1^{2n-2s+1}},\bC)\is
\mathcal H_{x'}^k\IC(\cO'_{2^{s}1^{2n-2s}},\bC)
\eeq
for 
$x\in\mO_{2^j1^{2n-2j+1}}$, $x'\in\mO'_{2^j1^{2n-2j}}$ and $j\leq s$. 

We show that $t^i_{jk}=(t^i_{jk})'$ by induction on $j$. The case $j=i$ is clear.
So by induction, we can assume that $t_{j'k}^{i}=(t_{j'k}^{i})'$ holds for 
$j<j'\leq i$.  Then, for $x_j\in\mO_{2^j1^{2n-2j+1}}$, we have 
\beq\label{decomposition sigma}
(\sigma_i)_*\bC[-]|_{x_j}=
\IC(\cO_{2^i1^{2n-2i+1}},\bC)|_{x_j}\oplus
\bigoplus_{j\leq j'<i} 
\IC(\cO_{2^{j'}1^{2n-2j'+1}},\bC^{t_{j'k}^i})[\pm k]|_{x_j}.
\eeq
Since the stalk
$\IC(\cO_{2^i1^{2n-2i+1}},\bC)|_{x_j}$ is concentrated in degree 
$<-\dim\mO_{2^j1^{2n-2j+1}}$ 
and \linebreak
$\IC(\cO_{2^{j}1^{2n-2j+1}},\bC^{t_{jk}^i})[-k]$
is concentrated in degree $\geq-\dim\mO_{2^j1^{2n-2j+1}}$, 
the decomposition in (\ref{decomposition sigma})
implies the multiplicity numbers
$t_{jk}^i$ are uniquely determined by
the cohomology of $\sigma_i^{-1}(x_j)$ and the stalks 
$\IC(\cO_{2^{j'}1^{2n-2j'+1}},\bC^{t_{j'k}^i})|_{x_j}$
for $j<j'<i$. The numbers $(t_{jk}^i)'$ are determined, in the same manner, by the corresponding data.
Now since $t_{j'k}^i=(t_{j'k}^i)'$ for $j'>j$ (by induction)
the 
desired equality $t_{jk}^i=(t_{jk}^i)'$ follows from 
(\ref{equality for stalks 1)}).
\end{proof}

\section{Cohomology of Fano varieties}
\label{sec-Fano}
In this section we  compute the cohomology of Fano varieties of 
$k$-planes in  the smooth complete intersection of two quadrics in $\bP^{2n}$. We  denote these Fano varieties by $\on{Fano}_{k}^{2n}$.  Note that $\on{Fano}_{0}^{{2n}}$ is the smooth complete intersection of two quadrics in $\bP^{2n}$.

Let $\on{Gr}(k,n)$ denote the Grassmannian variety of $k$-dimensional subspaces in $\bC^n$. Let
$$g_{k,n}(q):=\sum \dim\,H^{2l}(\on{Gr}(k,n),\bC)\,q^l=\frac{\prod_{l=n-k+1}^n(1-q^l)}{\prod_{l=1}^k(1-q^l)}$$
be the Poincare polynomial of  $\on{Gr}(k,n)$.

Recall the monodromy representations  $ L_i$ which were defined in \S\ref{sec-the local systems}.
 The cohomology of the Fano variety $\on{Fano}_{i-1}^{2n}$ is described as follows.

 \begin{thm}\label{thm-fano}
We have
\beq\label{odd coho}
H^{2k+1}(\on{Fano}_{i-1}^{2n},\bC)=0,
\eeq
\beq\label{even coho}
H^{2k}(\on{Fano}_{i-1}^{2n},\bC)\cong\bigoplus_{j=0}^i L_j^{\oplus M_i(k,j)},
\eeq
{where $M_i(k,j)$ is the coefficient of $q^{k-j(n-i)}$ in $g_{i-j,2n-i-j}(q)$.}
\end{thm}

\subsection{Fano varieties and resolutions for $\bar\cO_{2^i1^{2n-2i+1}}$}
We start with the following simple observation, which is a 
direct consequence of Theorem \ref{matching}.

Let $\pi:\Sigma\rightarrow\fg^{rs}_1$ be a family of smooth projective varieties over $\fg_1^{rs}$ and let 
$R^k\pi_*\bC$ be the corresponding local system on $\fg^{rs}_1$. Suppose that 
the Fourier transform of $\IC(\fg_1,R^k\pi_*\bC)$ is supported on $\bar{\cO}_{2^n1}$ and 
is given by 
\[\mathfrak F(\IC(\fg_1,R^k\pi_*\bC))=\bigoplus_{j=0}^n \IC(\cO_{2^j1^{2n-2j+1}},\bC^{ m_{kj}}).\]
Then the cohomology of the fiber $\Sigma_x:=\pi^{-1}(x)$ over $x\in\fg_1^{rs}$ 
satisfies 
\[H^k(\Sigma_x,\bC)\cong\bigoplus_{j=0}^n L_j^{\oplus m_{kj}}.\]
Moreover, the isomorphism above is compatible 
with the monodromy actions.

Let us apply this observation to the following situation. Consider  the maps  
\beqn
\check\sigma_i:\{(x,0\subset V_i\subset V_i^\p\subset \bC^{2n+1})\,|\,x\in\Lg_1,\ xV_i\subset V_i^\p\}\to\Lg_1.
\eeqn
Note that for $x\in\Lg_1^{rs}$, we have
$
\check\sigma_i^{-1}(x)\cong\on{Fano}_{i-1}^{2n},
$
 the Fano variety of $(i-1)$-planes in the smooth complete intersection of two quadrics $Q(v)=0$ and $\langle xv,v\rangle_Q=0$ in $\bP^{2n}$.
 
 Let us consider $\pi_i=\check{\sigma}_i|_{\check{\sigma}_i^{-1}(\fg_1^{rs})}$, which is a smooth family of Fano varieties, and consider the corresponding local system
$R^k\pi_{i*}\bC$. Recall that we have Reeder's resolutions $\sigma_i$  for $\bar\cO_{2^i1^{2n-2i+1}}$ (see  \S\ref{resolution}) and 
 \beq\label{decomposition for rho_i}
\sigma_{i*}\bC[i(2n+1-i)]=\bigoplus_{j=0}^i\bigoplus_{k=0}^{2(i-j)(n-i)} \IC(\cO_{2^{j}1^{2n-2j+1}},\bC^{t^{i}_{jk}})[\pm k].
\eeq
Since the Fourier transform of $\IC(\cO_{2^{j}1^{2n-2j+1}},\bC)$  
is  supported on all of $\fg_1$ for all $j$ (Theorem \ref{matching}), the 
equation
\beqn 
\fF(\check\sigma_{i*}\bC[-])\cong{\sigma}_{i*}\bC[-]
\eeqn
implies that
\begin{eqnarray*}\label{key emu}
&&\fF(\check\sigma_{i*}\bC[-])\cong\bigoplus_{k=0}^{4i(n-i)}\fF(\IC(\fg_1,R^k\pi_{i*}\bC)[-k+2i(n-i)])\\
&&\quad\cong
\bigoplus_{j=0}^i\bigoplus_{k=0}^{2(i-j)(n-i)} \IC(\cO_{2^{j}1^{2n-2j+1}},\bC^{t^{i}_{jk}})[\pm k].\nonumber
\end{eqnarray*}
here 
$
2i(n-i)=\dim\check{\sigma}_i^{-1}(x)=\dim\on{Fano}_{i-1}^{2n},
$ for $x\in\Lg_1^{rs}$.

Hence we see that 
$\fF(\IC(\fg_1,R^k\pi_{i*}\bC))$ is supported on $\bar\cO_{2^i1^{2n-2i+1}}$, and 
has the form 
\[\fF(\IC(\fg_1,R^k\pi_{i*}\bC))=\bigoplus_{j=0}^i \IC(\cO_{2^j1^{2n-2j+1}},\bC^{t^i_{j,|2i(n-i)-k|}}).\]
So by the observation above and Theorem \ref{matching}, we deduce that the cohomology of the Fano varieties $\on{Fano}_{i-1}^{2n}$
is given by 
\beq\label{eqn coho of fano odd}
H^k(\on{Fano}_{i-1}^{2n},\bC)\cong\bigoplus_{j=0}^i L_j^{\oplus t^i_{j,|2i(n-i)-k|}}.
\eeq

\subsection{The numbers $t_{jk}^{i}$}

In this subsection let us again use the notation $(t^{i}_{jk})_n$ for the numbers $(t^i_{jk})$ in \eqref{decomposition for rho_i} to indicate that the ambient symmetric pair is $(SL(2n+1),SO(2n+1))$.

\begin{lemma}\label{odd vanish}
We have $(t^i_{jk})_n=0$ for odd $k$.
\end{lemma}
\begin{proof}
 In the decomposition \eqref{decomposition for rho_i}, we take the stalk $\mH^{l}_{x_j}$ on both sides  for odd $l$, where $x_j\in\cO_{2^j1^{2n+1-2j}}$. Since $i(2n-i+1)$ is even and $H^{\text{odd}}(\sigma_i^{-1}(x_j),\bC)=0$, we have that
$\mH^{l}_{x_j}\sigma_{i*}\bC[i(2n+1-i)]=0$ for all odd $l$. Suppose that there exists $k$ odd such that $t^i_{jk}\neq 0$, then there exists an odd $k$ such that $\mH^{l}_{x_j}\IC(\cO_{2^{j}1^{2n-2j+1}},\bC^{t^{i}_{jk}})[\pm k]\neq 0$ (note that $\mH^{-j(2n-j+1)}_{x_j}\IC(\cO_{2^{j}1^{2n-2j+1}},\bC)\neq 0$). This is a contradiction. Thus the lemma is proved.
\end{proof}

Recall from~\eqref{lem-decomposition numbers reduction} that we have $(t_{l,k}^i)_n=(t_{l-j,k}^{i-j})_{n-j}$ for $j\leq l< i$. This implies, in particular, that $(t^i_{jk})_n=(t^{i-j}_{0,k})_{n-j}$ for $j\geq 1$. Since $(t_{i,k}^i)_n=\delta_{0,k}$, the determination of $(t^{i}_{jk})_n$ are reduced to that of $(t^i_{0,k})_n$. In the following we describe how to determine the latter numbers inductively. The dimensions of the stalks $\mH^{k}_0\IC(\cO_{2^i1^{2n+1-2i}},\bC)$ can be determined simultaneously (in \S\ref{ssec-fj} we determine these dimensions directly, see \eqref{dim stalks}). 

Recall that $\mH^{\text{odd}}_0\IC(\cO_{2^{j}1^{2n-2j+1}},\bC)=0$ (see Corollary \ref{coro odd dimen vanishing}). Note also that \linebreak$\mH^{2k}_0\IC(\cO_{2^{j}1^{2n-2j+1}},\bC)\neq 0$ implies that $-\dim\cO_{2^{j}1^{2n-2j+1}}\leq 2k\leq -1$. Let us write 
$$m_j=(\dim\cO_{2^{j}1^{2n-2j+1}})/2=j(2n-j+1)/2,$$
$$f_j(q)=\sum_{k=-m_j}^{-1}(\dim \mH^{2k}_0\IC(\cO_{2^{j}1^{2n-2j+1}},\bC))\,q^k,$$
$$
\text{ and }\ \ og_{j,2n+1}(q)=\sum_{k=0}^{d_j}\dim H^{2k}(\on{OGr}(j,2n+1),\bC)\,q^k,$$
where $d_j=\dim\on{OGr}(j,2n+1)=j(4n-3j+1)/2$. The polynomials $og_{j,2n+1}(q)$ are known, i.e. 
\beqn
og_{i,2n+1}(q)=\frac{(1-q^{2(n-i+1)})(1-q^{2(n-i+2)})\cdots(1-q^{2n})}{(1-q)(1-q^2)\cdots(1-q^i)}.
\eeqn
Note that $\sigma_i^{-1}(0)\cong\on{OGr}(i,2n+1)$. In view of Lemma \ref{odd vanish}, the decomposition \eqref{decomposition for rho_i} implies that
\beq\label{inductive formula}
og_{i,2n+1}(q)\,q^{-i(2n-i+1)/2}=f_i(q)+\sum_{j=1}^{i-1}f_j(q)\,\sum_{k=0}^{(i-j)(n-i)}(t^i_{j,2k})_n\,q^{\pm k}+\sum_{k=0}^{i(n-i)}(t^i_{0,2k})_n\,q^{\pm k}.
\eeq
Note that $f_0(q)=1$. It is easy to check using \eqref{inductive formula} that 
\beqn
(t^1_{0,2k})_n=1\text{ for }0\leq k\leq n-1,\ (t^1_{0,2k})_n=0\text{ otherwise, and }f_1(q)=q^{-n}. 
\eeqn
This completes the determination of the numbers $(t^i_{jk})_n$ and $f_i(q)$'s for $n=1$, and for all $n$ and $i=1$. By induction on $n$, we can assume that the numbers $(t^i_{jk})_{n'}$ have been determined for all $n'<n$. This implies that $(t^i_{jk})_n$ for $j\geq 1$ have been determined. We determine now the numbers $(t^i_{0k})_n$ by induction on $i$. We can assume that all $f_{i'}(q)$, $i'<i$, has been determined. Note that $f_i(q)$ is concentrated in negative degrees. Thus we can determine the numbers $(t^i_{0,2k})_n$ from \eqref{inductive formula} and then determine $f_i(q)$. We have shown that
\beq\label{uniquely determined}
\text{the equations \eqref{inductive formula} determine $f_i(q)$'s and $(\sum_{k=0}^{i(n-i)}(t^i_{0,2k})_n\,q^{\pm k})$'s uniquely}.
\eeq

\subsection{The functions $f_i(q)$} \label{ssec-fj}In fact, the functions $f_i(q)$ can be determined directly making use of 
our identification of $\mH^k_0\IC(\cO_{2^i1^{2n+1-2i}},\bC)$ with $\mH^k_0\IC(\cO'_{2^i1^{2n-2i}},\bC)$ (see Theorem \ref{SO=SP}), where $\cO'_{2^i1^{2n-2i}}\subset\cN_{\mathfrak{sp}(2n)}$, and the classical Springer correspondence.
\begin{lemma}
We have that
\beq\label{eqn-fj}
f_i(q)=q^{-i(2n-i+1)/2}\,g_{[i/2],n}(q^2),
\eeq
where $[i/2]$ is the integer part of $i/2$. Namely
\beq\label{dim stalks}
\begin{gathered}
\dim\mathcal{H}^{k-i(2n-i+1)}_0\IC({\cO}_{2^i\,1^{2n-2i+1}},\bC)=\dim H^{k/2}(\on{Gr}([\frac{i}{2}],n),\bC)\text{ if }k\equiv 0 
\ (\on{mod} 4),\\
 \dim\mathcal{H}^{k-i(2n-i+1)}_0\IC({\cO}_{2^i\,1^{2n-2i+1}},\bC)=0\text{ otherwise}.
 \end{gathered}
 \eeq
\end{lemma}
\begin{proof}
We have
\begin{eqnarray*}
&&\dim\mathcal{H}^{k}_0\IC({\cO}_{2^i\,1^{2n-2i+1}},\bC)=\dim\mathcal{H}^{k}_0\IC({\cO}_{2^i\,1^{2n-2i}}',\bC)\stackrel{\eqref{sp resolution}}=[V_{(\cO_{2^i1^{2n-2i}}',\bC)}:H^{k+2n^2}(\mB,\bC)],
\end{eqnarray*} 
where $V_{(\cO_{2^i1^{2n-2i}}',\bC)}$ is the representation of $W_n$ attached to the pair $(\cO_{2^i1^{2n-2i}}',\bC)$ under Springer correspondence (see \S\ref{sec-proof of proposition cc}). The numbers 
$[V_{(\cO_{2^i1^{2n-2i}}',\bC)}:H^{k+2n^2}(\mB,\bC)]$ are the so-called fake degrees and they have been computed explicitly by Lusztig in \cite{Lu4}. In fact, let us write
\beqn
P_i(q)=\sum[V_{(\cO_{2^i1^{2n-2i}}',\bC)}:H^{2k}(\mB,\bC)]\,q^k.
\eeqn
Using \eqref{spc-1} and \cite{Lu4}, we see that
\beqn
P_i(q)=q^{n^2-ni+i(i-1)/2}g_{[i/2],n}(q^2).
\eeqn
Now $\dim\mathcal{H}^{k-i(2n-i+1)}_0\IC({\cO}_{2^i\,1^{2n-2i+1}},\bC)$ is the coefficient of $q^{\frac{k-i(2n-i+1)+2n^2}{2}}$ in $P_i(q)$, which is the coefficient of $q^{k/2}$ in $g_{[i/2],n}(q^2)$. The equation \eqref{dim stalks} follows. 
\end{proof}

\subsection{Proof of  Theorem \ref{thm-fano}}

The equation \eqref{odd coho} in the theorem follows from \eqref{eqn coho of fano odd} and  Lemma \ref{odd vanish}. The equation \eqref{even coho} in the theorem follows from \eqref{eqn coho of fano odd} and the following statement about the numbers $t^i_{j,2k}$.

\begin{proposition}
We have
$$\sum_{k=0}^{(i-j)(n-i)}(t^{i}_{j,2k})_{n}q^{\pm k}=q^{-(i-j)(n-i)}g_{i-j,2n-i-j}(q).$$
\end{proposition}
In view of \eqref{uniquely determined},   
the proposition above follows from \eqref{inductive formula}, \eqref{eqn-fj}, and the following equation
\beq\label{poincare poly1}
og_{i,2n+1}(q)=
\sum_{j=0}^{i}q^{(i-j)(i-j+1)/2}\,g_{[j/2],n}(q^2)\,g_{i-j,2n-i-j}(q),\text{ where $g_{0,n}(q)=1$.}
\eeq

\begin{proof}[Proof of \eqref{poincare poly1}]
The proof given here was kindly supplied to us by Dennis Stanton.

Define $(A;q)_l:=\prod_{k=0}^{l-1}(1-Aq^k)$ and
$
\displaystyle{{n\brack j}_q:=\frac{(q;q)_n}{(q;q)_j(q;q)_{n-j}}.}
$
Let us write $$S=\sum_{j=0}^{i}q^{(i-j)(i-j+1)/2}\,g_{[j/2],n}(q^2)\,g_{i-j,2n-i-j}(q).$$
We have 
$$S=\sum_{j=0}^{[i/2]}q^{(i-2j-1)(i-2j)/2}\,{n\brack j}_{q^2}(\,{2n-i-2j-1\brack i-2j-1}_q+q^{i-2j}{2n-i-2j\brack i-2j}_q)
$$
$$=\frac{q^{\binom{i}{ 2}}\,(q;q)_{2n-i-1}}{(q;q)_{2n-2i}\,(q;q)_i}\,\sum_{j=0}^{[i/2]}\,(-1)^j\,\frac{(q^{-2n};q^2)_j\,(q^{-i};q)_{j}\,(q^{1-i};q)_j\,q^{j(j+2i-2n+4)}}{(q^2;q^2)_j\,(q^{-2n+i+1};q^2)_{j}\,(q^{-2n+i+2};q^2)_{j}}\,(1-q^{2n-4j}).$$
Consider the terminating very well-poised basic hypergeometric series which, by definition, is given by
$$
{}_6\phi_5\left(\begin{array}{cc}a,\,q^2a^{1/2},\,-q^2a^{1/2},\ b,\ c,\ q^{-2m}\\\ a^{1/2},\ -a^{1/2},\,aq^2/b,\ aq^2/c,\ aq^{2m+2}\end{array};\, q^2,\ \frac{aq^{2m+2}}{bc}\right)
$$
$$
=\sum_{j=0}^{\infty}\frac{(a;q^2)_j\,(q^2a^{1/2};q^2)_j\,(-q^2a^{1/2};q^2)_j\, (b;q^2)_j\,( c;q^2)_j\, (q^{-2m};q^2)_j}{(q^2;q^2)_j\, (a^{1/2};q^2)_j\,( -a^{1/2};q^2)_j\,(aq^2/b;q^2)_j\, (aq^2/c;q^2)_j\,( aq^{2m+2};q^2)_j}(\frac{aq^{2m+2}}{bc})^j.
$$
Note that if $m>0$, then $(q^{-2m};q^2)_j=0$ for $j>m$ and thus the  sum above is finite.  
Take 
$$a=q^{-2n};\  m=i/2\text{ and }b=q^{1-i}\text{ for even }i; \text{ and }m=(i-1)/2,\ b=q^{-i}\text{ for odd }i.$$
One checks that 
$$S=(1-q^{2n})\frac{q^{\binom{i}{ 2}}(q;q)_{2n-i-1}}{(q;q)_{2n-2i}\,(q;q)_i}\,\lim_{c\to\infty}{}_6\phi_5\left(\begin{array}{cc}a,\,q^2a^{1/2},\,-q^2a^{1/2},\ b,\ c,\ q^{-2m}\\\ a^{1/2},\ -a^{1/2},\,aq^2/b,\ aq^2/c,\ aq^{2m+2}\end{array};\, q^2,\ \frac{aq^{2m+2}}{bc}\right).
$$
Now use the summation formula (see \cite[Appendix (II.21)]{GR})
\beqn
{}_6\phi_5\left(\begin{array}{cc}a,\,q^2a^{1/2},\,-q^2a^{1/2},\ b,\ c,\ q^{-2m}\\\ a^{1/2},\ -a^{1/2},\,aq^2/b,\ aq^2/c,\ aq^{2m+2}\end{array};\, q^2,\ \frac{aq^{2m+2}}{bc}\right)=\frac{(aq^2;q^2)_m\,(aq^2/bc;q^2)_m}{(aq^2/b;q^2)_m\,(aq^2/c;q^2)_m}
\eeqn
we get
\beqn
\lim_{c\to\infty}{}_6\phi_5\left(\begin{array}{cc}a,\,q^2a^{1/2},\,-q^2a^{1/2},\ b,\ c,\ q^{-2m}\\\ a^{1/2},\ -a^{1/2},\,aq^2/b,\ aq^2/c,\ aq^{2m+2}\end{array};\, q^2,\ \frac{aq^{2m+2}}{bc}\right)=\frac{(aq^2;q^2)_m}{(aq^2/b;q^2)_m}.
\eeqn
Thus 
\beqn
S=(1-q^{2n})\,\frac{q^{\binom{i}{ 2}}\,(q;q)_{2n-i-1}\,(aq^2;q^2)_m}{(q;q)_{2n-2i}\,(q;q)_i\,(aq^2/b;q^2)_m}={og}_{i,2n+1}(q).
\eeqn
\end{proof}

\begin{example}[Cohomology of $\on{Fano}_1^{2n}$]

The cohomology of $\on{Fano}_1^{2n}$, the Fano variety of lines in the smooth complete intersection of two quadrics in $\bP^{2n}$,  
can be described as follows: $$H^{2k+1}(\on{Fano}_1^{2n},\bC)=0$$
\beqn
H^{2(4n-8-k)}(\on{Fano}_1^{2n},\bC)=H^{2k}(\on{Fano}_1^{2n},\bC)=\left\{\begin{array}{ll}\bC^{[\frac{k+2}{2}]}&\text{ if }0\leq k\leq n-3\\\bC^{[\frac{k+2}{2}]}\oplus L_1&\text{ if }n-2\leq k\leq 2n-5\\\bC^{[\frac{k+2}{2}]}\oplus L_1\oplus L_2&\text{ if }k=2n-4.\end{array}\right.
\eeqn

\end{example}


\begin{thebibliography}{999999}
\bibitem[A]{A} N. A'Campo, 
{\em Tresses, monodromie et le groupe symplectique}, 
Comment. Math. Helv. {\bf 54} (1979), no.~2, 318--327.

\bibitem[BG]{BG} M.~Bhargava and B.~H.~Gross, 
{\em The average size of the $2$-Selmer group of Jacobians of hyperelliptic curves having a rational Weierstrass point}, in: Automorphic Representations and $L$-Functions, pp.~23--91, 
Tata Inst. Fundam. Res. Stud. Math. {\bf 22}, Tata Inst. Fund. Res., Mumbai, (2013).
 
\bibitem[C]{C} C.~Casagrande, 
{\em Rank $2$ quasiparabolic vector bundles on $\bP^1$ and the variety of linear subspaces contained in two odd-dimensional quadrics}, 
Mathematische Zeitschrift {\bf 280} (2015), no.~3-4, 981--988.

\bibitem[CVX1]{CVX1} T.~H. Chen, K.~Vilonen, and T.~Xue, 
{\em Hessenberg varieties, intersections of quadrics, and the Springer correspondence},
Trans. Amer. Math. Soc. {\bf 373} (2020), no.~4, 2427--2461.

\bibitem[CVX2]{CVX2}
T.~H. Chen, K. Vilonen, and T. Xue,
{\em On the cohomology of Fano varieties and the Springer correspondence} (with an appendix by Dennis Stanton), Adv. in Math. {\bf 318} (2017), 515--533.

\bibitem[CVX3]{CVX3}
T.~H. Chen, K. Vilonen, and T. Xue, 
{\em Springer correspondence for the split symmetric pair in type A},
 Compos. Math. {\bf 154} (2018), no.~11, 2403--2425.


\bibitem[CLP]{CLP} C. De Concini, G. Lusztig, and C. Procesi, 
{\em Homology of the zero-set of a nilpotent vector field on a flag manifold},
 J. Amer. Math. Soc. {\bf 1} (1988), no.~1, 15--34.


\bibitem[EM]{EM} S.~Evens and I.~Mirkovi\'c, 
{\em Characteristic cycles for the loop Grassmannian and nilpotent orbits},
Duke Math. J. {\bf 97} (1999), no.~1, 109--126. 

\bibitem[GR]{GR} G.~Gasper and M. Rahman, 
 Basic Hypergeometric Series,  second edition, Encyclopedia of Mathematics and its Applications {\bf 96}, Cambridge University Press, Cambridge, (2004). 


\bibitem[He]{Hes} W.~H.~Hesselink, 
{\em Desingularizations of varieties of nullforms},
 Invent. Math. {\bf 55} (1979), no.~2, 141--163. 


\bibitem[KS]{KaS} M.~Kashiwara and P.~Schapira, 
Sheaves on Manifolds, Grundlehren der Mathematischen Wissenschaften {\bf 292},
 Springer-Verlag, Berlin, (1990). 

\bibitem[KaS]{KS} N.~M.~Katz and P.~Sarnak, 
Random Matrices, Frobenius Eigenvalues, and Monodromy,
 American Mathematical Society Colloquium Publications {\bf 45},
 American Mathematical Society, Providence, RI, (1999). 


\bibitem[KR]{KR} B.~Kostant and S.~Rallis,
{\em Orbits and representations associated with symmetric spaces},
 Amer. J. Math. {\bf 93} (1971), 753--809.


\bibitem[L1]{Lu} G.~Lusztig, 
{\em Intersection cohomology complexes on a reductive group},
 Invent. Math. {\bf 75} (1984), no.~2, 205--272. 

\bibitem[L2]{Lu4} G.~Lusztig, 
{\em Irreducible representations of finite classical groups},
 Invent. Math. {\bf 43} (1977), no.~2, 125--175. 

\bibitem[L3]{Lu2} G.~Lusztig, 
{\em An induction theorem for Springer's representations},
in: Representation Theory of Algebraic Groups and Quantum Groups, pp.~253--259, 
Adv. Stud. Pure Math. {\bf 40}, Math. Soc. Japan, Tokyo, (2004).

\bibitem[R]{R} M.~Reeder, 
{\em  Desingularizations of some unstable orbit closures},
 Pacific J. Math. {\bf 167} (1995), no.~2, 327--343.

\bibitem[Re]{Re} M.~Reid, 
{\em The complete intersection of two or more quadrics},
 PhD Thesis, Trinity College, Cambridge. (1972).

\bibitem[S]{S} J.~Sekiguchi, 
{\em The nilpotent subvariety of the vector space associated to a symmetric pair},
 Publ. Res. Inst. Math. Sci. {\bf 20} (1984), no.~1, 155--212.


\end{thebibliography}
\end{document}